\documentclass[a4paper]{amsart}

\usepackage{amssymb,amsmath, amsfonts}
\usepackage{amsmath, amsfonts, amsthm , amssymb, amscd}
\usepackage[all]{xy}
\usepackage{graphicx}
\usepackage{pst-grad} 
\usepackage{booktabs}

\usepackage{amsthm}
\usepackage{url}

\newcommand\str{\rightarrow}
\newcommand\mj{\mbox{\bf 1}}

 \newtheorem{thm}{Theorem}[section]
 \newtheorem{prop}[thm]{Proposition}
 \newtheorem{lem}[thm]{Lemma}
 \newtheorem{cor}[thm]{Corollary}
 \newtheorem{exm}[thm]{Example}
 \newtheorem{rem}[thm]{Remark}
 \newtheorem{que}[thm]{Question}
 \newtheorem{dfn}{Definition}[section]
\numberwithin{equation}{section}

\title{A simple permutoassociahedron}

\author{Djordje Barali\'{c}, Jelena Ivanovi\'{c} and Zoran Petri\'{c}}
\address{ \scriptsize{Mathematical Institute SANU\\ Knez Mihailova 36, p.f.\ 367\\
        11001 Belgrade, Serbia}}
\email{djbaralic@mi.sanu.ac.rs}
\address{\scriptsize{University of Belgrade, Faculty of Architecture\\ Bulevar Kralja Aleksandra 73/II\\ 11000 Belgrade, Serbia}}
\email{jelena.s.ivanovic@gmail.com}
\address{\scriptsize{Mathematical Institute SANU\\ Knez Mihailova 36, p.f.\ 367\\
        11001 Belgrade, Serbia}}
\email{zpetric@mi.sanu.ac.rs}

\date{}

\begin{document}

\begin{abstract}
In the early 1990s, a family of combinatorial CW-complexes named
permutoassociahedra was introduced by Kapranov, and it was
realised by Reiner and Ziegler as a family of convex polytopes.
The polytopes in this family are ``hybrids'' of permutohedra and
associahedra. Since permutohedra and associahedra are simple, it
is natural to search for a family of simple permutoassociahedra,
which is still adequate for a topological proof of Mac Lane's
coherence. This paper presents such a family.

\vspace{.3cm}


\vspace{.5ex}

\noindent {\small {\it Keywords$\,$}: associativity,
    commutativity, coherence, simple polytopes, geometric realisation}
\end{abstract}

\maketitle

\noindent {\small \emph{The paper is dedicated to Jim Stasheff
whose work is the cornerstone of this beautiful mathematics.}}

\section{Introduction}

A geometry of commutativity and associativity interaction is
demonstrated by polytopes called permutoassociahedra. This family
of polytopes arises as a ``hybrid'' of two more familiar
families---permutohedra and associahedra. Kapranov, \cite{K93},
defined a combinatorial CW-complex $KP_n$ (denoted by
$\mathbf{KPA}_{n-1}$ in \cite{RZ94}), whose vertices correspond to
all possible complete bracketings of permuted products of $n$
letters. He showed that for every $n$, this CW-complex is an
$(n-1)$-ball, and found a three dimensional polytope that realises
$KP_4$. Reiner and Ziegler, \cite{RZ94}, realised all these
CW-complexes as convex polytopes.

However, for every $n\geq 4$, the polytope $KP_n$ is not simple.
Our goal is to define a new family of polytopes
$\mathbf{PA}_{n-1}$, having the same vertices as $KP_n$, which are
all simple, and which may still serve as a topological proof of
Mac Lane's coherence.

The combinatorics of the family $\mathbf{PA}_n$ is introduced in
\cite{P14}. This family is called \emph{permutohedron-based
associahedra} in that paper, and for the sake of simplicity, we
call it here \emph{simple permutoassociahedra}. From
\cite[Proposition~10.1]{P14}, it follows that there exist simple
polytopes that realise these combinatorially defined objects.

Our polytopes $\mathbf{PA}_n$ belong to a family that generalises
polytopes called \emph{nestohedra} or \emph{hypergraph polytopes}.
The work of Fulton and MacPherson, \cite{FMP94}, De Concini and
Procesi, \cite{DeCP95}, Stasheff and Shnider, \cite{S97}, Gaiffi,
\cite{Ga03}, \cite{Ga04}, Feichtner and Kozlov, \cite{FK04},
Feichtner and Sturmfels, \cite{FS05}, Carr and Devadoss,
\cite{CD06}, Postnikov, Reiner and Williams, \cite{PRW08},
Postnikov, \cite{P09}, Do\v sen and the third author, \cite{DP10},
and of many others established the family of nestohedra as an
important tool joining algebra, combinatorics, geometry, logic,
topology and some other fields of mathematics.

We pay attention in particular to 0, 1 and 2-faces of polytopes
$\mathbf{PA}_n$. These faces suggest a choice of generators and
corresponding equations relevant for building a symmetric monoidal
category freely generated by a set of objects. All this is
important for Mac Lane's coherence. However, the main result of
the paper is an explicit realisation of the family $\mathbf{PA}_n$
given by systems of inequalities representing halfspaces in
$\mathbf{R}^{n+1}$. Our realisation has its roots in the
realisation of associahedra and cyclohedra given in \cite{S97}.

Throughout the text, the subset relation is denoted by
$\subseteq$, while the \emph{proper} subset relation is denoted by
$\subset$.

\section{Varying the combinatorics by generators}

In this section we try to formalise the main differences between
the approach of \cite{K93}, leading to the family of
permutoassociahedra $KP_n$, and our approach, leading to a family
of simple permutoassociahedra. Roughly speaking, the main
difference is the choice of arrows that generate symmetry in a
symmetric monoidal category. The reader interested mainly in
combinatorics of the new family of polytopes, may skip this
section and just take a look at the commutative diagrams given
below.

The notion of symmetric monoidal categories (not under this name)
is given, and a coherence result is proven by Mac Lane in
\cite{ML63}. (For the definition of symmetric monoidal categories
we refer to \cite[XI.1]{ML71}). The original proof of coherence
relies on a strictification of the monoidal structure of such a
category followed by the standard presentation of symmetric groups
by generators and relations (for details see \cite[\S3.3,
Chapter~5]{DP04}).

Since a coherence result of a class of categories deals with
canonical arrows, it is natural to formulate it in terms of a
category in this class freely generated by a set of objects. This
idea is suggested by Voreadou, \cite{V77}, and it is fairly used
in \cite{DP04}. There is no explicit referring to this approach in
\cite{K93}---however, one may reconstruct the instant one step
proof of Mac Lane's coherence, suggested by the author, through
the symmetric monoidal category freely generated by an infinite
set of objects. Also, this provides us a possibility to compare
the two families of permutoassociahedra.

Given the set $\omega=\{0,1,2,\ldots\}$ of generating objects,
what is a symmetric monoidal category $\mathbf{SM}$ with the
strict unit $I$, freely generated by $\omega$? For the set of
objects of this category, there is no doubt---this set contains
$I$, and all the other objects are built out of elements of
$\omega$ with the help of binary operation denoted here by
$\cdot$, which is the object part of the required bifunctor. For
example, $(2\cdot((4\cdot 2)\cdot 0))$ is an object of this
category and we omit, as usual, the outermost parentheses denoting
this object by $2\cdot((4\cdot 2)\cdot 0)$.

The arrows of $\mathbf{SM}$ are equivalence classes of terms. One
can choose here different languages to build these terms, and
hence to have different equivalence relations, generated by
commutative diagrams, that produce the arrows of $\mathbf{SM}$.

Our reconstruction of the approach of \cite{K93} is that the
\emph{terms} are built inductively in the following manner.

\vspace{2ex}

($\mj$.1) For an object $A$, the \emph{identity} $\mj_A:A\str A$
is an $\mj$-term;

\vspace{1ex}

($\mj$.2) for $f$ and $g$, $\mj$-terms, $(f\cdot g)$ is an
$\mj$-term.

\vspace{2ex}

($\alpha$.1) For objects $A$, $B$ and $C$ distinct from $I$,
$\alpha_{A,B,C}:A\cdot(B\cdot C)\str(A\cdot B)\cdot C$ and
$\alpha^{-1}_{A,B,C}:(A\cdot B)\cdot C\str A\cdot (B\cdot C)$ are
$\alpha$-terms;

\vspace{1ex}

($\alpha$.2) for $f$ an $\alpha$-term and $g$ an $\mj$-term,
$(f\cdot g)$ and $(g\cdot f)$ are $\alpha$-terms.

\vspace{2ex}

($\tau$.1) For $p,q\in\omega$, $\tau_{p,q}:p\cdot q\str q\cdot p$
is a $\tau$-term;

\vspace{1ex}

($\tau$.2) for $f$  a $\tau$-term and $g$ an $\mj$-term, $(f\cdot
g)$ and $(g\cdot f)$ are $\tau$-terms.

\vspace{2ex}

($t$.1) Every $\mj$-term, $\alpha$-term and $\tau$-term is a term;

\vspace{1ex}

($t$.2) if $f:A\str B$ and $g:B\str C$ are terms, then $g\circ
f:A\str C$ is a term.

\vspace{2ex}

\noindent For example,
\[
\alpha^{-1}_{2\cdot 4,1,0}\circ(\alpha_{2,4,1}\cdot
\mj_0)\circ((\mj_2\cdot \tau_{1,4})\cdot \mj_0)\circ
(\alpha^{-1}_{2,1,4}\cdot\mj_0):((2\cdot 1)\cdot 4)\cdot 0\str
(2\cdot 4)\cdot(1\cdot 0)
\]
is a term.

The identities are neutrals for composition. Moreover,
$\mj_A\cdot\mj_B=\mj_{A\cdot B}$ and the following diagrams
commute for $f$, $g$ and $f\cdot(g\cdot h)$ being $\alpha$-terms
or $\tau$-terms.

\begin{center}
\begin{picture}(120,50)

\put(0,40){\makebox(0,0){$A\cdot B$}}
\put(0,10){\makebox(0,0){$A\cdot B'$}}
\put(120,40){\makebox(0,0){$A'\cdot B$}}
\put(120,10){\makebox(0,0){$A'\cdot B'$}}

\put(0,30){\vector(0,-1){10}} \put(120,30){\vector(0,-1){10}}
\put(25,40){\vector(1,0){70}} \put(25,10){\vector(1,0){70}}

\put(-15,25){\small\makebox(0,0){$\mj\cdot g$}}
\put(135,25){\small\makebox(0,0){$\mj\cdot g$}}
\put(60,46){\small\makebox(0,0){$f\cdot\mj$}}
\put(60,4){\small\makebox(0,0){$f\cdot\mj$}}

\put(200,25){\makebox(0,0){(1)}}

\end{picture}
\end{center}

\begin{center}
\begin{picture}(120,50)

\put(0,40){\makebox(0,0){$A\cdot(B\cdot C)$}}
\put(0,10){\makebox(0,0){$A\cdot(B\cdot C)$}}
\put(120,40){\makebox(0,0){$(A\cdot B)\cdot C$}}
\put(120,10){\makebox(0,0){$(A\cdot B)\cdot C$}}

\put(0,30){\vector(0,-1){10}} \put(120,30){\vector(0,-1){10}}
\put(30,40){\vector(1,0){60}} \put(30,10){\vector(1,0){60}}
\put(90,35){\vector(-3,-1){60}}

\put(-10,25){\small\makebox(0,0){$\mj$}}
\put(130,25){\small\makebox(0,0){$\mj$}}
\put(60,45){\small\makebox(0,0){$\alpha$}}
\put(60,5){\small\makebox(0,0){$\alpha$}}
\put(55,29){\small\makebox(0,0){$\alpha^{-1}$}}

\put(200,25){\makebox(0,0){(2)}}

\end{picture}
\end{center}

\begin{center}
\begin{picture}(120,110)

\put(60,100){\makebox(0,0){$A\cdot(B\cdot(C\cdot D))$}}
\put(0,70){\makebox(0,0){$A\cdot((B\cdot C)\cdot D)$}}
\put(0,40){\makebox(0,0){$(A\cdot(B\cdot C))\cdot D$}}
\put(60,10){\makebox(0,0){$((A\cdot B)\cdot C)\cdot D$}}
\put(120,55){\makebox(0,0){$(A\cdot B)\cdot (C\cdot D)$}}

\put(50,90){\vector(-3,-1){30}} \put(0,60){\vector(0,-1){10}}
\put(20,30){\vector(3,-1){30}} \put(70,90){\vector(2,-1){40}}
\put(110,40){\vector(-2,-1){40}}

\put(20,88){\small\makebox(0,0){$\mj\cdot \alpha$}}
\put(-10,55){\small\makebox(0,0){$\alpha$}}
\put(20,22){\small\makebox(0,0){$\alpha\cdot\mj$}}
\put(97,82){\small\makebox(0,0){$\alpha$}}
\put(97,28){\small\makebox(0,0){$\alpha$}}

\put(200,55){\makebox(0,0){(3)}}

\end{picture}
\end{center}

\begin{center}
\begin{picture}(120,50)

\put(0,40){\makebox(0,0){$A\cdot(B\cdot C)$}}
\put(0,10){\makebox(0,0){$A'\cdot(B'\cdot C')$}}
\put(120,40){\makebox(0,0){$(A\cdot B)\cdot C$}}
\put(120,10){\makebox(0,0){$(A'\cdot B')\cdot C'$}}

\put(0,30){\vector(0,-1){10}} \put(120,30){\vector(0,-1){10}}
\put(30,40){\vector(1,0){60}} \put(30,10){\vector(1,0){60}}

\put(-25,25){\small\makebox(0,0){$f\cdot(g\cdot h)$}}
\put(145,25){\small\makebox(0,0){$(f\cdot g)\cdot h$}}
\put(60,45){\small\makebox(0,0){$\alpha$}}
\put(60,5){\small\makebox(0,0){$\alpha$}}

\put(200,25){\makebox(0,0){(4)}}

\end{picture}
\end{center}

\begin{center}
\begin{picture}(120,55)

\put(10,40){\makebox(0,0){$p\cdot q$}}
\put(10,10){\makebox(0,0){$p\cdot q$}}
\put(110,40){\makebox(0,0){$q\cdot p$}}

\put(10,30){\vector(0,-1){10}} \put(30,40){\vector(1,0){60}}
\put(90,35){\vector(-3,-1){60}}

\put(0,25){\small\makebox(0,0){$\mj$}}
\put(60,45){\small\makebox(0,0){$\tau$}}
\put(60,18){\small\makebox(0,0){$\tau$}}

\put(200,25){\makebox(0,0){(5)}}

\end{picture}
\end{center}

\begin{center}
\begin{picture}(120,145)

\put(25,130){\makebox(0,0){$p\cdot (q\cdot r)$}}
\put(-10,106){\makebox(0,0){$(p\cdot q)\cdot r$}}
\put(-30,82){\makebox(0,0){$(q\cdot p)\cdot r$}}
\put(-30,58){\makebox(0,0){$q\cdot(p\cdot r)$}}
\put(-10,34){\makebox(0,0){$q\cdot(r\cdot p)$}}
\put(25,10){\makebox(0,0){$(q\cdot r)\cdot p$}}

\put(95,130){\makebox(0,0){$p\cdot (r\cdot q)$}}
\put(130,106){\makebox(0,0){$(p\cdot r)\cdot q$}}
\put(150,82){\makebox(0,0){$(r\cdot p)\cdot q$}}
\put(150,58){\makebox(0,0){$r\cdot(p\cdot q)$}}
\put(130,34){\makebox(0,0){$r\cdot(q\cdot p)$}}
\put(95,10){\makebox(0,0){$(r\cdot q)\cdot p$}}

\put(-2,28){\vector(2,-1){22}} \put(-26,53){\vector(1,-1){13}}
\put(-30,76){\vector(0,-1){12}} \put(-13,100){\vector(-1,-1){13}}
\put(20,122){\vector(-2,-1){22}} \put(122,28){\vector(-2,-1){22}}
\put(146,52){\vector(-1,-1){13}} \put(150,76){\vector(0,-1){12}}
\put(133,100){\vector(1,-1){13}} \put(100,122){\vector(2,-1){22}}

\put(50,130){\vector(1,0){20}} \put(50,10){\vector(1,0){20}}

\put(60,115){\makebox(0,0){}} \put(60,5){\makebox(0,0){}}

\put(-2,20){\small\makebox(0,0){$\alpha$}}
\put(-32,45){\small\makebox(0,0){$\mj\cdot \tau$}}
\put(-40,70){\small\makebox(0,0){$\alpha^{-1}$}}
\put(-32,95){\small\makebox(0,0){$\tau\cdot\mj$}}
\put(-2,120){\small\makebox(0,0){$\alpha$}}

\put(122,20){\small\makebox(0,0){$\alpha$}}
\put(152,45){\small\makebox(0,0){$\mj\cdot \tau$}}
\put(162,70){\small\makebox(0,0){$\alpha^{-1}$}}
\put(152,95){\small\makebox(0,0){$\tau\cdot\mj$}}
\put(122,120){\small\makebox(0,0){$\alpha$}}

\put(60,137){\small\makebox(0,0){$\mj\cdot \tau$}}
\put(60,3){\small\makebox(0,0){$\tau\cdot\mj$}}

\put(200,70){\makebox(0,0){(6)}}

\end{picture}
\end{center}

All these equations are assumed to \emph{hold in contexts}. For
example, the equation (4) in the context $\cdot D$ reads

\begin{center}
\begin{picture}(120,50)

\put(0,40){\makebox(0,0){$(A\cdot(B\cdot C))\cdot D$}}
\put(-5,10){\makebox(0,0){$(A'\cdot(B'\cdot C'))\cdot D$}}
\put(120,40){\makebox(0,0){$((A\cdot B)\cdot C)\cdot D$}}
\put(125,10){\makebox(0,0){$((A'\cdot B')\cdot C')\cdot D$}}

\put(0,30){\vector(0,-1){10}} \put(120,30){\vector(0,-1){10}}
\put(40,40){\vector(1,0){40}} \put(40,10){\vector(1,0){40}}

\put(-35,25){\small\makebox(0,0){$(f\cdot(g\cdot h))\cdot\mj$}}
\put(155,25){\small\makebox(0,0){$((f\cdot g)\cdot h)\cdot\mj$}}
\put(60,45){\small\makebox(0,0){$\alpha\cdot\mj$}}
\put(60,5){\small\makebox(0,0){$\alpha\cdot\mj$}}

\end{picture}
\end{center}

Out of the category $\mathbf{SM}$, one can build a
(\emph{pseudo})\emph{graph} $\mathcal{G}$ (in the terminology of
\cite{H69}) whose \emph{vertices} are the objects of
$\mathbf{SM}$, while $A$ and $B$ are joined by an \emph{edge} in
$\mathcal{G}$, when there is an $\alpha$-term, or a $\tau$-term
$f:A\str B$. Since every such term has an inverse, the edges are
undirected. There are no multiple edges. However, this graph
contains loops since, for example, there is an edge joining
$p\cdot p$ with itself, witnessed by the term $\tau_{p,p}$.

Every connected component of $\mathcal{G}$ contains the objects
with the same number of occurrences of elements of $\omega$.
Hence, the object $I$, as well as every element of $\omega$, makes
a singleton connected component of $\mathbf{SM}$. For the
coherence question, we are interested in \emph{diversified}
connected components, i.e.\ the connected components containing
the objects of $\mathbf{SM}$ without repeating the elements of
$\omega$. In particular, for every $n\in\omega$, we are interested
in the connected component $\mathcal{G}_n$ of $\mathcal{G}$
containing the object of the form
\[
0\cdot(1\cdot(\ldots\cdot n)\ldots).
\]
It is straightforward to see that $\mathcal{G}_n$ does not contain
loops.

In order to establish the symmetric monoidal coherence (cf.\
\cite[\S 3.3,Chapter~5]{DP04}), it suffices to show that for every
$n\in\omega$, every diagram involving the objects of
$\mathcal{G}_n$ commutes in $\mathbf{SM}$. A topological proof of
this fact consists in finding an $n$-dimensional CW-ball, whose
0-cells are the vertices of $\mathcal{G}_n$, 1-cells are the edges
of $\mathcal{G}_n$, and 2-cells correspond to quadrilaterals (1)
and (4), pentagons (3) and dodecagons (6) given above. This all is
done by Kapranov in \cite{K93}.

Our approach is a bit different. We construct a category
$\mathbf{SM}$* with the same objects as $\mathbf{SM}$, for which
we show that is, up to renaming of arrows, the same
as~$\mathbf{SM}$.

To fix the notation we define the \emph{left}, \emph{right} and
\emph{middle} \emph{contexts} in the following way. The symbol
$\square$ is a left and a right context. If $L$ is a left and $R$
is a right context, while $A$ is an object of $\mathbf{SM}$*, then
$(A\cdot L)$ is a left context, $(R\cdot A)$ is a right context,
and $(R\cdot L)$ is a middle context.

If $L$, $R$ and $M$ are left, right and middle contexts,
respectively, and $a,b\in\omega$, then $La$ denotes the object of
$\mathbf{SM}$* obtained by replacing the occurrence of $\square$
in $L$ by $a$, and we define the objects $bR$ and $bMa$ in the
same way. It is obvious that for every object $A$ of
$\mathbf{SM}$* different from $I$, there exists a unique left
context $L$ and a unique $a\in\omega$, namely the rightmost
occurrence of an element of $\omega$ in $A$, such that $A$ is
$La$. The same holds for right contexts and if $A$ contains at
least one $\cdot$, the same holds for middle contexts.

Our definition of \emph{terms} of $\mathbf{SM}$* is as for
$\mathbf{SM}$ save that the clauses ($\tau$.1) and ($\tau$.2) are
replaced by the clause

\vspace{2ex}

($\sigma$) For objects $La$ and $bR$, $\sigma_{La,bR}:La\cdot
bR\str Lb\cdot aR$ is a $\sigma$-term,

\vspace{2ex}

\noindent and in ($t$.1), ``$\tau$-term'' is replaced by
``$\sigma$-term''.

For example,
\[
\sigma_{2\cdot 1,4\cdot 0}\circ \alpha^{-1}_{2\cdot
1,4,0}:((2\cdot 1)\cdot 4)\cdot 0\str (2\cdot 4)\cdot(1\cdot 0)
\]
is a term.

To define the arrows of $\mathbf{SM}$*, we use the equivalence
relation on terms, which is generated by the commutative diagrams
(1), (2), (3) and (4) (save that $f$, $g$ and $f\cdot(g\cdot h)$
are $\alpha$-terms) and the following commutative diagrams
involving~$\sigma$'s.

\begin{center}
\begin{picture}(120,50)

\put(0,40){\makebox(0,0){$La\cdot bR$}}
\put(0,10){\makebox(0,0){$La\cdot bR$}}
\put(120,40){\makebox(0,0){$Lb\cdot aR$}}

\put(0,32){\vector(0,-1){14}} \put(30,40){\vector(1,0){60}}
\put(90,35){\vector(-3,-1){60}}

\put(-8,25){\small\makebox(0,0){$\mj$}}
\put(60,45){\small\makebox(0,0){$\sigma$}}
\put(60,18){\small\makebox(0,0){$\sigma$}}

\put(200,25){\makebox(0,0){(7)}}

\end{picture}
\end{center}

For the diagram (8), it is assumed that $f:La\str L'a$ and
$g:bR\str bR'$ are $\alpha$-terms and that $f^a_b:Lb\str L'b$ and
$g^b_a:aR\str aR'$ are corresponding $\alpha$-terms obtained by
substituting $b$ for the distinguished occurrence of $a$ in $f$,
and vice versa for $g$.

\begin{center}
\begin{picture}(200,60)(0,-5)

\put(0,40){\makebox(0,0){$La\cdot bR$}}
\put(0,10){\makebox(0,0){$Lb\cdot aR$}}
\put(70,40){\makebox(0,0){$L'a\cdot bR$}}
\put(70,10){\makebox(0,0){$L'b\cdot aR$}}

\put(0,32){\vector(0,-1){14}} \put(70,32){\vector(0,-1){14}}
\put(20,40){\vector(1,0){30}} \put(20,10){\vector(1,0){30}}

\put(-7,25){\small\makebox(0,0){$\sigma$}}
\put(77,25){\small\makebox(0,0){$\sigma$}}
\put(35,47){\small\makebox(0,0){$f\cdot \mj$}}
\put(35,3){\small\makebox(0,0){$f^a_b\cdot \mj$}}

\put(240,25){\makebox(0,0){(8)}}

\put(130,40){\makebox(0,0){$La\cdot bR$}}
\put(130,10){\makebox(0,0){$Lb\cdot aR$}}
\put(200,40){\makebox(0,0){$La\cdot bR'$}}
\put(200,10){\makebox(0,0){$Lb\cdot aR'$}}

\put(130,30){\vector(0,-1){10}} \put(200,30){\vector(0,-1){10}}
\put(150,40){\vector(1,0){30}} \put(150,10){\vector(1,0){30}}

\put(123,25){\small\makebox(0,0){$\sigma$}}
\put(207,25){\small\makebox(0,0){$\sigma$}}
\put(165,47){\small\makebox(0,0){$\mj\cdot g$}}
\put(165,3){\small\makebox(0,0){$\mj\cdot g^b_a$}}

\end{picture}
\end{center}

\begin{center}
\begin{picture}(120,95)

\put(10,82){\makebox(0,0){$La\cdot (bMc\cdot dR)$}}
\put(-20,58){\makebox(0,0){$(La\cdot bMc)\cdot dR$}}
\put(-20,34){\makebox(0,0){$(La\cdot bMd)\cdot cR$}}
\put(10,10){\makebox(0,0){$La\cdot (bMd\cdot cR)$}}

\put(110,82){\makebox(0,0){$Lb\cdot (aMc\cdot dR)$}}
\put(140,58){\makebox(0,0){$(Lb\cdot aMc)\cdot dR$}}
\put(140,34){\makebox(0,0){$(Lb\cdot aMd)\cdot cR$}}
\put(110,10){\makebox(0,0){$Lb\cdot (aMd\cdot cR)$}}

\put(-13,27){\vector(2,-1){22}} \put(-20,52){\vector(0,-1){12}}
\put(10,75){\vector(-2,-1){22}} \put(133,27){\vector(-2,-1){22}}
\put(140,52){\vector(0,-1){12}} \put(110,75){\vector(2,-1){22}}

\put(50,82){\vector(1,0){20}} \put(50,10){\vector(1,0){20}}

\put(-20,21){\small\makebox(0,0){$\alpha^{-1}$}}
\put(-30,45){\small\makebox(0,0){$\sigma$}}
\put(-15,71){\small\makebox(0,0){$\alpha$}}

\put(135,21){\small\makebox(0,0){$\alpha^{-1}$}}
\put(150,45){\small\makebox(0,0){$\sigma$}}
\put(135,71){\small\makebox(0,0){$\alpha$}}

\put(60,89){\small\makebox(0,0){$\sigma$}}
\put(60,3){\small\makebox(0,0){$\sigma$}}

\put(200,50){\makebox(0,0){(9)}}

\end{picture}
\end{center}

\begin{center}
\begin{picture}(120,145)

\put(20,130){\makebox(0,0){$La\cdot (b\cdot cR)$}}
\put(-10,106){\makebox(0,0){$(La\cdot b)\cdot cR$}}
\put(-30,82){\makebox(0,0){$(La\cdot c)\cdot bR$}}
\put(-30,58){\makebox(0,0){$La\cdot (c\cdot bR)$}}
\put(-10,34){\makebox(0,0){$Lc\cdot (a\cdot bR)$}}
\put(20,10){\makebox(0,0){$(Lc\cdot a)\cdot bR$}}

\put(100,130){\makebox(0,0){$Lb\cdot (a\cdot cR)$}}
\put(130,106){\makebox(0,0){$(Lb\cdot a)\cdot cR$}}
\put(150,82){\makebox(0,0){$(Lb\cdot c)\cdot aR$}}
\put(150,58){\makebox(0,0){$Lb\cdot (c\cdot aR)$}}
\put(130,34){\makebox(0,0){$Lc\cdot (b\cdot aR)$}}
\put(100,10){\makebox(0,0){$(Lc\cdot b)\cdot aR$}}

\put(-2,28){\vector(2,-1){22}} \put(-26,53){\vector(1,-1){13}}
\put(-30,76){\vector(0,-1){12}} \put(-13,100){\vector(-1,-1){13}}
\put(20,122){\vector(-2,-1){22}} \put(122,28){\vector(-2,-1){22}}
\put(146,52){\vector(-1,-1){13}} \put(150,76){\vector(0,-1){12}}
\put(133,100){\vector(1,-1){13}} \put(100,122){\vector(2,-1){22}}

\put(50,130){\vector(1,0){20}} \put(50,10){\vector(1,0){20}}

\put(-2,20){\small\makebox(0,0){$\alpha$}}
\put(-32,45){\small\makebox(0,0){$\sigma$}}
\put(-40,70){\small\makebox(0,0){$\alpha^{-1}$}}
\put(-32,95){\small\makebox(0,0){$\sigma$}}
\put(-2,120){\small\makebox(0,0){$\alpha$}}

\put(122,20){\small\makebox(0,0){$\alpha$}}
\put(152,45){\small\makebox(0,0){$\sigma$}}
\put(162,70){\small\makebox(0,0){$\alpha^{-1}$}}
\put(152,95){\small\makebox(0,0){$\sigma$}}
\put(122,120){\small\makebox(0,0){$\alpha$}}

\put(60,137){\small\makebox(0,0){$\sigma$}}
\put(60,3){\small\makebox(0,0){$\sigma$}}

\put(200,70){\makebox(0,0){(10)}}

\end{picture}
\end{center}
The equations (1), (2), (3) and (4) are assumed to hold in
contexts, while the contexts make no sense for the equations (7),
(8), (9) and (10).

It is more or less straightforward to define the $\sigma$-terms by
$\tau$-terms and $\alpha$-terms, and to show by induction that the
equations (7), (8), (9) and (10) hold in $\mathbf{SM}$ for defined
$\sigma$'s. The equations (5), (4) with $f\cdot(g\cdot h)$ being a
$\tau$-term, (1) with $f$, $g$ being $\tau$-terms and (6),
respectively, are essential for this proof. However, if one is not
interested in a new topological proof of Mac Lane's coherence,
then this coherence result instantly delivers the commutativity of
the above diagrams in $\mathbf{SM}$.

For the other direction, the $\tau$-terms are easily defined in
terms of $\sigma$-terms and $\alpha$-terms. Our family of
polytopes (see Section~\ref{realisation}) will guarantee that all
the equations of $\mathbf{SM}$ hold in $\mathbf{SM}$* for defined
$\tau$-terms. Hence, $\mathbf{SM}$ and $\mathbf{SM}$* are, up to
renaming of arrows, the same.

Let $\mathcal{G}^*$ be a (pseudo)graph with the same
\emph{vertices} as $\mathcal{G}$, while $A$ and $B$ are joined by
an \emph{edge} in $\mathcal{G}^*$, when there is an $\alpha$-term,
or a $\sigma$-term $f:A\str B$. As in the case of $\mathcal{G}$,
we denote by $\mathcal{G}^*_n$ the connected component of
$\mathcal{G}^*$ containing the object
\[
0\cdot(1\cdot(\ldots\cdot n)\ldots).
\]
It is evident that $\mathcal{G}^*_n$ is $n$-regular, i.e.\ it has
exactly $n$ edges incident to every vertex.

A graph $G$ is a \emph{graph of a polytope} $P$ when the vertices
and edges of $P$ are in one-to-one correspondence with the
vertices and edges of $G$, and this correspondence respects the
incidence relation. By a result of Blind and Mani-Levitska,
\cite{BM87} (see also Kalai's elegant proof, \cite{K88}), one can
conclude that if there is an $n$-dimensional simple polytope whose
graph is $\mathcal{G}^*_n$, then all such polytopes are
combinatorially equivalent. Our goal is to geometrically realise
such a polytope for every $n$, and to show that its 2-faces
correspond to quadrilaterals (1), (4) and (8), pentagons (3),
octagons (9) and dodecagons (10).

\section{A two-fold nested set construction}\label{two-fold}

A family of simplicial complexes that are closely related to the
face lattices of a family of polytopes, which we call simple
permutoassociahedra, is introduced in this section. Our
construction of this family of simplicial complexes is an example
of the two-fold construction of complexes of nested sets for
simplicial complexes, which is thoroughly investigated in
\cite{P14}. Since we are here interested in a single family of
complexes, there is no reason to repeat the theory of nested set
complexes in its full generality and we present only the facts
necessary for our construction. We refer to \cite{P14} for a
detailed treatment of the theory.

An \textit{abstract simplicial complex} on a set $T$ is a
non-empty collection $K$ of finite subsets of $T$ closed under
taking inclusions, i.e.\ if $\sigma \in K$ and $\tau\subseteq
\sigma$, then $\tau \in K$. The elements of $K$ are called
\emph{simplices}. It is assumed that $\{v\}$ is a simplex of $K$
for every element $v\in T$. An abstract simplicial complex $K$ is
\textit{finite} when the number of its simplices is finite,
otherwise it is \textit{infinite}. In this contribution we
consider only finite abstract simplicial complexes, and we call
them shortly \emph{simplicial complexes}.

A simplex $\tau\in K$ is a \textit{face} of a simplex $\sigma \in
K$ if $\tau\subseteq \sigma$. A simplex $\sigma \in K$ is called
\textit{maximal} when it is not a face of some other simplex from
$K$. If $V_1,\ldots, V_m$ are the maximal simplices of $K$, than
$K$ has a presentation as
\[
K=\mathcal{P}(V_1)\cup\ldots\cup \mathcal{P}(V_m).
\]
Observe that the maximal simplices of a simplicial complex $K$ are
incomparable (with respect to $\subseteq$) finite sets.

\begin{dfn}\label{d1}
A collection $\mathcal{B}$ of non-empty subsets of a finite set
$V$ containing all singletons $\{v\}, v\in V$ and satisfying that
for any two sets $S_1, S_2\in \mathcal{B}$ such that $S_1\cap S_2
\neq\emptyset$, their union $S_1\cup S_2$ also belongs to
$\mathcal{B}$, is called a \textit{building set} of
$\mathcal{P}(V)$.

Let $K$ be a simplicial complex and let $V_1,\ldots,V_m$ be the
maximal simplices of $K$. A collection $\mathcal{B}$ of some
simplices of $K$ is called a \emph{building set} of $K$ when for
every $i=1,\ldots,m$, the collection
$$\mathcal{B}^{V_i}=\mathcal{B} \cap \mathcal{P}(V_i)$$ is a
building set of $\mathcal{P}(V_i)$.
\end{dfn}
\noindent The notion of a building set of a finite-meet
semilattice introduced in \cite[Definition~2.2]{FK04} reduces to
this notion when the semilattice is a simplicial complex (cf.\
\cite[Section~3]{P14}).

For a family of sets $N$, we say that $\{X_1,\ldots,X_m\}\subseteq
N$ is an $N$-\emph{antichain}, when $m\geq 2$ and $X_1,\ldots,X_m$
are mutually incomparable with respect to $\subseteq$.

\begin{dfn}\label{d2}
Let $\mathcal{B}$ be a building set of a simplicial complex $K$.
We say that $N\subseteq \mathcal{B}$ is a \emph{nested set} with
respect to $\mathcal{B}$, when the union of every $N$-antichain is
an element of $K-\mathcal{B}$.
\end{dfn}
\noindent This definition is in accordance with
\cite[Definition~2.7]{FK04} (cf.\ \cite[Section~3]{P14}).

Now we proceed to the main construction of the paper. We start
with the set $X=\{0,\ldots,n\}$, $n\geq 1$, i.e.\ the ordinal
$n+1$. (The case when $n=0$ is trivial.) Let $C_0$ be the
simplicial complex $\mathcal{P}(X)-\{X\}$, i.e.\ the family of
subsets of $X$ with at most $n$ elements, and let
$\mathcal{B}_0=C_0-\{\emptyset\}$. In the literature, $C_0$ is
also known as the \emph{boundary complex} $\partial \Delta^n$ of
the abstract $n$-simplex $\Delta^n$, and obviously $\mathcal{B}_0$
is a building set of the simplicial complex $C_0$ according to
Definition~\ref{d1}.

A set $N\subseteq \mathcal{B}_0$ such that the union of every
$N$-antichain belongs to $C_0-\mathcal{B}_0$ is called
0-\emph{nested}. According to Definition~\ref{d2}, every 0-nested
set is a nested set with respect to $\mathcal{B}_0$. Since
$C_0-\mathcal{B}_0=\{\emptyset\}$, every two members of a 0-nested
set are comparable.


It is evident that a subset of a 0-nested set is a 0-nested
set---thus the family of all 0-nested sets makes a simplicial
complex, which we denote by $C_1$. The maximal 0-nested sets are
of the form
\[
\bigl\{ \{i_n,\ldots,i_1\},\ldots,\{i_n,i_{n-1}\},\{i_n\} \bigr\},
\]
where $i_1,\ldots,i_n$ are mutually distinct elements of $X$. Such
a 0-nested set corresponds to the permutation
\[
i_0 i_1\ldots i_n
\]
of $X$, where $\{i_0\}=X-\{i_1,\ldots,i_n\}$. For example, when
$X=\{0,1,2,3\}$,
\[
\bigl\{ \{1,0,3\},\{1,0\},\{1\} \bigr\}
\]
is a maximal 0-nested set that corresponds to the permutation
$2301$ (see the picture at the end of Section~\ref{realisation}).

We say that a polytope $P$ (\emph{geometrically}) \emph{realises}
a simplicial complex $K$, when the semilattice obtained by
removing the bottom (the empty set) from the face lattice of $P$
is isomorphic to $(K,\supseteq)$. The $n$-dimensional
permutohedron (cf.\ \cite[Definition~4.1]{M66}) realises the
simplicial complex $C_1$. To prove this, one has to rely on
\cite[Proposition~3.4]{P14} and well known facts about nested set
complexes related to permutohedra (cf.\ \cite{FS05}, \cite{CD06},
\cite{P09} and \cite{DP10}).

In this one-to-one correspondence, the vertices of the
permutohedron correspond to the maximal 0-nested sets. Hence,
there are $(n+1)!$ maximal 0-nested sets. Also, the facets of the
permutohedron correspond to the singleton nested sets of the form
$\{A\}$, where $A\in \mathcal{B}_0$. It is more intuitive to
denote the facets of this permutohedron by ordered partitions of
$X$ with two blocks. With singleton nested sets we just economise
a little bit---our nested set $\{A\}$ corresponds to the ordered
partition whose first block is $X-A$ and the second is~$A$.

For a maximal 0-nested set
\[
\bigl\{ \{i_n,\ldots,i_1\},\ldots,\{i_n,i_{n-1}\},\{i_n\} \bigr\},
\]
consider the following \emph{path graph} with $n$ vertices and
$n-1$ edges:
\[
\{i_n,\ldots,i_1\}-\ldots-\{i_n,i_{n-1}\}-\{i_n\}.
\]
We say that a set of vertices of this graph is \emph{connected},
when this is the set of vertices of a connected subgraph of this
graph.

Let $\mathcal{B}_1\subseteq C_1$ be the family of all sets of the
form
\[
\bigl\{
\{i_{k+l},\ldots,i_k,\ldots,i_1\},\ldots,\{i_{k+l},\ldots,i_k,i_{k-1}\},\{i_{k+l},\ldots,i_k\}
\bigr\},
\]
where $1\leq k\leq k+l\leq n$ and $i_1,\ldots,i_{k+l}$ are
mutually distinct elements of $X$, i.e.\ $\mathcal{B}_1$ is the
set of all non-empty connected sets of vertices of the path graphs
that correspond to all maximal 0-nested sets. According to
Definition~\ref{d1}, we have that $\mathcal{B}_1$ is a building
set of the simplicial complex $C_1$.

\begin{rem}
Every member of $C_1-\mathcal{B}_1$ is of the form
$\{A_1,\ldots,A_r\}$, where
\[
2\leq r\leq n-1 \;\&\; X\supset A_1\supset\ldots\supset A_r \;\&\;
\exists i \:|A_i-A_{i+1}| \geq 2.
\]
\end{rem}

\begin{exm}\label{e1}
For $X=\{0,1,2\}$, we have that $\mathcal{B}_1$ is
\[
\Bigl\{
\bigl\{\{0\}\bigr\},\;\bigl\{\{1\}\bigr\},\;\bigl\{\{2\}\bigr\},\;
\bigl\{\{0,1\}\bigr\},\; \bigl\{\{0,2\}\bigr\},\;
\bigl\{\{1,2\}\bigr\} ,
\]
\[\bigl\{\{0,1\},\{0\}\bigr\},\; \bigl\{\{0,2\},\{0\}\bigr\},\;
\bigl\{\{0,1\},\{1\}\bigr\},
\]
\[
\bigl\{\{1,2\},\{1\}\bigr\},\; \bigl\{\{0,2\},\{2\}\bigr\},\;
\bigl\{\{1,2\},\{2\}\bigr\} \Bigr\}.
\]
\end{exm}

A set $N\subseteq \mathcal{B}_1$ is 1-\emph{nested} when the union
of every $N$-antichain belongs to $C_1-\mathcal{B}_1$. According
to Definition~\ref{d2}, every 1-nested set is a nested set with
respect to $\mathcal{B}_1$. Again, it is evident that the family
of all 1-nested sets makes a simplicial complex, which we denote
by $C$. As a corollary of the following proposition, we have that
$C$ is a flag complex.

\begin{prop}\label{p2}
A set $N\subseteq \mathcal{B}_1$ is 1-nested when the union of
every two incomparable elements of $N$ belongs to
$C_1-\mathcal{B}_1$.
\end{prop}

\begin{proof}
Let $\mathcal{A}$ be an $N$-antichain. By our assumption, every
two different members of $\mathcal{A}$ are of the form
\[
\{A_{i1},\ldots,A_{ir}\}\quad{\rm and}\quad
\{A_{j1},\ldots,A_{jq}\},
\]
with $A_{i1}\supset\ldots\supset A_{ir}\supset
A_{j1}\supset\ldots\supset A_{jq}$, and
\[
|A_{ik}-A_{i(k+1)}|=1=|A_{jk}-A_{j(k+1)}|,\quad |A_{ir}-
A_{j1}|\geq 2.
\]
Hence $\mathcal{A}$ can be linearly ordered in a natural way, and
we may conclude that $\bigcup\mathcal{A}$ is a subset of a maximal
0-nested set, which means that it is in $C_1$. The ``gaps'' arose
by provisos $|A_{ir}- A_{j1}|\geq 2$ guarantee that
$\bigcup\mathcal{A}$ is not in~$\mathcal{B}_1$.
\end{proof}

\begin{exm}\label{e8}
For $X=\{0,1,2\}$, we have the following 12 maximal 1-nested sets:
\[\begin{array}{lll}
\Bigl\{\bigl\{\{1,2\},\{2\}\bigr\},\bigl\{\{1,2\}\bigr\}\Bigr\}, &
\Bigl\{\bigl\{\{1,2\},\{2\}\bigr\},\bigl\{\{2\}\bigr\}\Bigr\}, &
\Bigl\{\bigl\{\{0,2\},\{2\}\bigr\},\bigl\{\{2\}\bigr\}\Bigr\},
\\[1ex]
\Bigl\{\bigl\{\{0,2\},\{2\}\bigr\},\bigl\{\{0,2\}\bigr\}\Bigr\}, &
\Bigl\{\bigl\{\{0,2\},\{0\}\bigr\},\bigl\{\{0,2\}\bigr\}\Bigr\}, &
\Bigl\{\bigl\{\{0,2\},\{0\}\bigr\},\bigl\{\{0\}\bigr\}\Bigr\},
\\[1ex]
\Bigl\{\bigl\{\{0,1\},\{0\}\bigr\},\bigl\{\{0\}\bigr\}\Bigr\}, &
\Bigl\{\bigl\{\{0,1\},\{0\}\bigr\},\bigl\{\{0,1\}\bigr\}\Bigr\}, &
\Bigl\{\bigl\{\{0,1\},\{1\}\bigr\},\bigl\{\{0,1\}\bigr\}\Bigr\},
\\[1ex]
\Bigl\{\bigl\{\{0,1\},\{1\}\bigr\},\bigl\{\{1\}\bigr\}\Bigr\}, &
\Bigl\{\bigl\{\{1,2\},\{1\}\bigr\},\bigl\{\{1\}\bigr\}\Bigr\}, &
\Bigl\{\bigl\{\{1,2\},\{1\}\bigr\},\bigl\{\{1,2\}\bigr\}\Bigr\},
\end{array}
\]
and $C$ is the union of power sets of these sets.
\end{exm}

\begin{exm}\label{e2}
If $X=\{0,1,2,3\}$ and $M=\{\{1,0,3\},\{1,0\},\{1\}\}$, then
\[
N=\Bigl\{M,\bigl\{\{1\}\bigr\},\bigl\{\{1,0,3\}\bigr\}\Bigr\},
\quad
\Bigl\{M,\bigl\{\{1,0,3\},\{1,0\}\bigr\},\bigl\{\{1,0,3\}\bigr\}\Bigr\},
\]
\[
\Bigl\{M,\bigl\{\{1,0\},\{1\}\bigr\},\bigl\{\{1\}\bigr\}\Bigr\},
\quad\quad
\Bigl\{M,\bigl\{\{1,0,3\},\{1,0\}\bigr\},\bigl\{\{1,0\}\bigr\}\Bigr\},
\]
\[
{\rm \it
and}\quad\Bigl\{M,\bigl\{\{1,0\},\{1\}\bigr\},\bigl\{\{1,0\}\bigr\}\Bigr\},
\]
are 5 out of 120 maximal 1-nested sets.
\end{exm}

In Example~\ref{e2}, every maximal 1-nested set, except $N$, is
such that all of its members are comparable (it contains no
antichains), while $N$ contains the $N$-antichain
$\{\{\{1\}\},\{\{1,0,3\}\}\}$ whose union is
$\{\{1\},\{1,0,3\}\}\not\in \mathcal{B}_1$.

If $N$ is a 1-nested set, then there is a maximal 0-nested set $M$
such that every member of $N$ is a subset of $M$, i.e.\
$N\subseteq \mathcal{B}_1^M=\mathcal{B}_1\cap \mathcal{P}(M)$. In
this case, we say that $N$ is \emph{derived} from $M$. In
Example~\ref{e2}, the five displayed 1-nested sets are derived
from $M$. If $N$ is a maximal 1-nested set, then there is unique
such $M$, namely the maximal 0-nested set contained in $N$. This
also holds for every 1-nested set that contains a maximal 0-nested
set. Otherwise, if $N$ does not contain a maximal 0-nested set,
then there is not a unique such $M$.

\begin{rem}\label{r1}
Every maximal 0 and 1-nested set is of cardinality $n$.
\end{rem}

\begin{rem}\label{r5}
Every element of $\mathcal{B}_1$ is contained in a maximal
1-nested set.
\end{rem}

For a maximal 0-nested set $M$, let $C^M$ be the \emph{link} of
$M$ in $C$, i.e.\ $C^M$ is the set of subsets not containing $M$
of the elements of $C$ that contain $M$. This simplicial complex
is related to the $(n-1)$-dimensional associahedron (Stasheff
polytope, see \cite{S63} and \cite{S97}) in the same way as the
simplicial complex $C_1$ has been related to the $n$-dimensional
permutohedron.

The vertices of the $(n-1)$-dimensional associahedron correspond
to all the \[c_n=\frac{1}{n+1}\binom{2n}{n}\] different ways of
bracketing a string of $n+1$ letters. The number $c_n$ is the
$n$th \emph{Catalan} number. Two vertices are adjacent if they
correspond to a single application of the associative law. In
general, $k$-faces of the associahedron are in bijection with the
set of correct bracketings of a string of $n+1$ letters with $n-k$
pairs of brackets. Two vertices lie in the same $k$-face if and
only if the corresponding complete bracketings could be reduced,
by removing $k$ pairs of brackets, to the same bracketing of the
string of $n+1$ letters with $n-k$ pairs of brackets. Figure
\ref{s:asn} depicts the $(n-1)$-dimensional associahedra for
$n=2,3$ and $4$.

\begin{figure}[h!h!h!]
\includegraphics[width=0.8\textwidth]{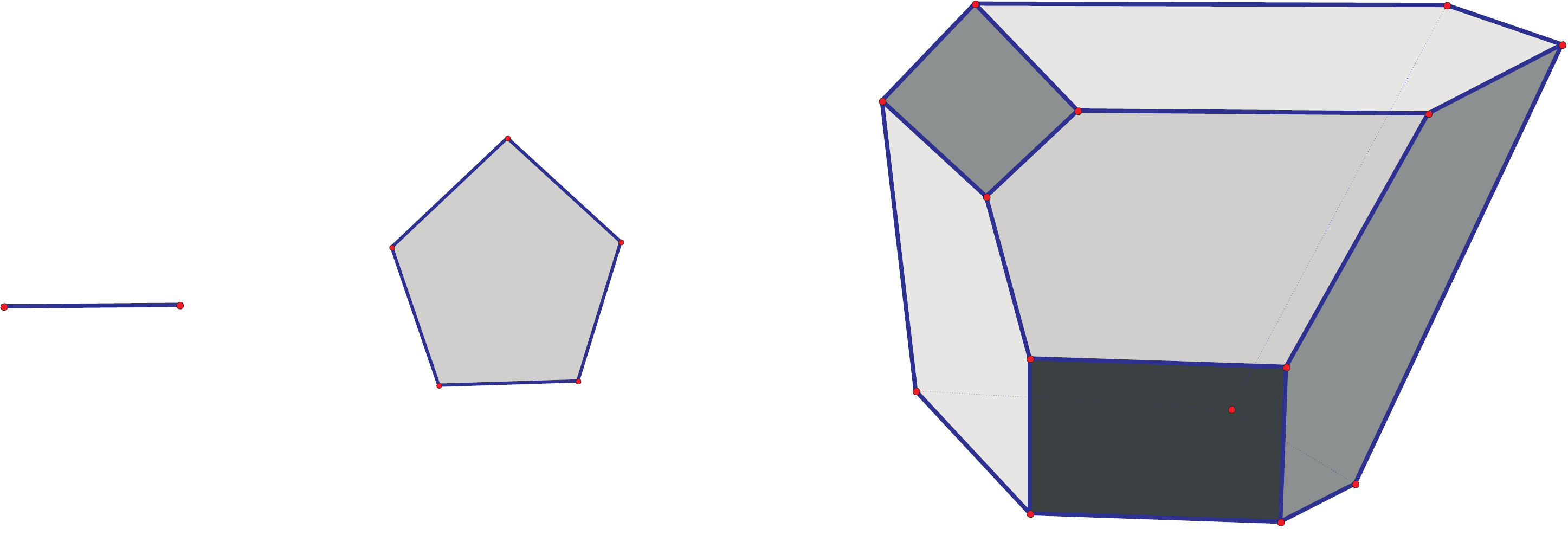}
\caption{$K_3$, $K_4$ and $K_5$} \label{s:asn}
\end{figure}

The $(n-1)$-dimensional associahedron realises $C^M$. This is
again proved by relying on well known facts about nested set
complexes related to associahedra (cf.\ \cite{FS05}, \cite{CD06},
\cite{P09} and \cite{DP10}). Hence, there are $c_n$ maximal
1-nested sets derived from a maximal 0-nested set. From this, one
concludes that there are all in all
\[
\frac{(2n)!}{n!}
\]
maximal 1-nested sets. A correspondence between the maximal
1-nested sets and complete bracketings of permuted products of the
elements of $\{0,\ldots,n\}$ is given in Section~\ref{faces}.

The main question related to the simplicial complex $C$ is the
following.
\begin{que}\label{q1}
Is there an $n$-dimensional realisation of $C$?
\end{que}
The positive answer to this question is given in
Section~\ref{realisation}. But first, in Section~\ref{faces}, we
show that if such a polytope exists, then its graph is
$\mathcal{G}^*_n$ and the 2-faces of this polytope correspond to
the diagrams (1), (4), (8), (3), (9) and~(10).

\section{The 0, 1 and 2-faces of $C$}\label{faces}

In order to establish the correspondence announced above, we
investigate the elements of the simplicial complex $C$ that should
correspond to the vertices, edges and two dimensional faces of a
polytope, which provides the positive answer to Question~\ref{q1}.
Respectively, we call these elements of $C$ the 0-faces, 1-faces
and 2-faces.

\setcounter{subsection}{-1}

\subsection{0-faces}

The 0-\emph{faces} of $C$ are the maximal 1-nested sets. Every
0-face corresponds to a unique complete bracketing of a permuted
product of the elements of $\{0,\ldots,n\}$ in the following way.
Let $V$ be a 0-face derived from
\[
M=\bigl\{\{i_n,\ldots,i_1\},\ldots,\{i_n,i_{n-1}\},\{i_n\}\bigr\}.
\]
We start with $n+1$ place holders separated by $n$ dots, which are
labelled by the elements of $M$ in the descending (with respect to
$\subseteq$) order:
\[
\underline{\hspace{1em}} \stackrel{\{i_n,\ldots,i_1\}}{\cdot}
\underline{\hspace{1em}} \ldots \underline{\hspace{1em}}
\stackrel{\{i_n,i_{n-1}\}}{\cdot} \underline{\hspace{1em}}
\stackrel{\{i_n\}}{\cdot} \underline{\hspace{1em}}.
\]
Each label denotes the elements of $X$ to the right of the
labelled dot. Hence we have the product:
\[
i_0 \stackrel{\{i_n,\ldots,i_1\}}{\cdot} i_1
\stackrel{\{i_n,\ldots,i_2\}}{\cdot} \ldots
\stackrel{\{i_n,i_{n-1}\}}{\cdot} i_{n-1}
\stackrel{\{i_n\}}{\cdot} i_n,
\]
where $\{i_0\}=X-\{i_1,\ldots,i_n\}$. Eventually, for every
element $A$ of $V$, one has to insert the pair of brackets that
include the labels contained in $A$.

\begin{exm}\label{e3}
For $X=\{0,1,2,3\}$, $M=\{\{1,0,3\},\{1,0\},\{1\}\}$, and
$V=\{M,\{\{1\}\},$\\ $\{\{1,0,3\}\}\}$, we have that $V$
corresponds to the following complete bracketing
\[
((2 \stackrel{\{1,0,3\}}{\cdot} 3 )\stackrel{\{1,0\}}{\cdot} (0
\stackrel{\{1\}}{\cdot} 1)).
\]
In the picture given at the end of Section~\ref{realisation}, the
0-face $V$ corresponds to the vertex incident to the pentagon
covering the vertex labelled by 2301, dodecagon covering the label
$\{1\}$ and dodecagon covering the label $\{0,1,3\}$.
\end{exm}

Also, for every complete bracketing $P$ of a product $i_0\cdot
i_1\cdot\ldots\cdot i_n$, one can find a unique 0-face of $C$ that
corresponds to $P$. Hence, the 0-faces of $C$ are in one-to-one
correspondence with the vertices of $\mathcal{G}^*_n$.

\begin{rem}\label{ziegler}
Since the face lattice of a simple polytope is completely
determined by the bipartite graph of incidence relation between
the vertex set and the set of facets, we have the following
alternative way to define the lattice structure, which may replace
our simplicial complex $C$. This is suggested to us by G.M.\
Ziegler.

As the vertex set take all the complete bracketings of permuted
products of $X=\{0,\ldots,n\}$, and as the set of facets take all
the ordered partitions of $X$ into $k$ blocks, $k\geq 2$, where
only the first and the last block are allowed to contain more than
one element. A vertex $V$, whose underlying permutation is $i_0
i_1\ldots i_n$, and a facet $F$ are incident when there is a pair
of corresponding brackets in $V$ with $i_j$ being the leftmost and
$i_k$ being the rightmost element in the scope of these brackets,
such that $F$ is the ordered partition of $X$ whose first block is
$\{i_0,\ldots,i_j\}$ followed by $k-j-1$ singleton blocks
$\{i_{j+1}\},\ldots,\{i_{k-1}\}$, and the last block is
$\{i_k,\ldots,i_n\}$.

If we assign to an element of $\mathcal{B}_1$ of the form
\[
\bigl\{ \{i_{k+l},\ldots,i_k,\ldots,i_1\},\ldots,
\{i_{k+l},\ldots,i_k,i_{k-1}\},\{i_{k+l},\ldots,i_k\} \bigr\},
\]
the ordered partition of $X$ whose first block is
$X-\{i_{k+l},\ldots,i_k,\ldots,i_1\}$, the second block is
$\{i_1\}$, and so on up to the penultimate block, which is
$\{i_{k-1}\}$, while the last block is $\{i_{k+l},\ldots,i_k\}$,
then the above analysis of 0-faces explains the incidence relation
given in this remark.
\end{rem}

\subsection{1-faces}

The 1-\emph{faces} of $C$ are obtained from the 0-faces by
removing a single element. There are two main types of 1-faces of
$C$; a 1-face of the \emph{first type} contains a maximal 0-nested
set, and a 1-face of the \emph{second type} is obtained from a
0-face by removing its maximal 0-nested set.

If $E$ is a 1-face of the first type, and
\[
M=\{\{i_n,\ldots,i_1\},\ldots,\{i_n,i_{n-1}\},\{i_n\}\}
\]
is the maximal 0-nested set contained in $E$, then a unique
bracketing of the product $i_0\cdot i_1\cdot\ldots\cdot i_n$, with
one pair of brackets omitted, is assigned to $E$ in the same way
as a complete bracketing has been assigned to a 0-face. There are
exactly two 0-faces that include a 1-face $E$ of the first type.
They correspond to two possible insertions of pairs of brackets in
the bracketing corresponding to~$E$. Hence, every 1-face of the
first type corresponds to an $\alpha$-term.

\begin{exm}\label{e4}
For $X=\{0,1,2,3\}$, $M=\{\{1,0,3\},\{1,0\},\{1\}\}$, and
$E=\{M,\{\{1\}\}\}$, we have that $E$ corresponds to the following
bracketing
\[
(2 \stackrel{\{1,0,3\}}{\cdot} 3 \stackrel{\{1,0\}}{\cdot} (0
\stackrel{\{1\}}{\cdot} 1)).
\]
The 0-face $V$ from Example~\ref{e3} and the 0-face
$V'=\{M,\{\{1,0\},\{1\}\},\{\{1\}\} \}$, which corresponds to the
complete bracketing
\[
(2 \cdot (3 \cdot (0 \cdot 1))),
\]
are the 0-faces that include $E$. In the picture given at the end
of Section~\ref{realisation}, the 1-face $E$ corresponds to the
edge incident to the pentagon covering the vertex labelled by 2301
and dodecagon covering the label $\{1\}$.
\end{exm}

As before, for every bracketing $P$ of a product $i_0\cdot
i_1\cdot\ldots\cdot i_n$, with one pair of brackets omitted, one
can find a unique 1-face of the first type that corresponds
to~$P$.

In order to describe the 1-faces of the second type, we define
some notions, whose more general form is introduced in
\cite{DP10}. A \emph{construction} of a path graph
\[
v_1-v_2-\ldots-v_n
\]
$n\geq 1$, is a subset of $\mathcal{P}(\{v_1,\ldots,v_n\})$
inductively defined as follows:

\vspace{2ex}

(1) if $n=1$, then $\{\{v_1\}\}$ is the construction of this graph
with a single vertex~$v_1$;

\vspace{1ex}

(2) if $n>1$, $1\leq j\leq n$ and $L$, $R$ are constructions of
the path graphs
\[
v_1-\ldots-v_{j-1}\quad{\rm and}\quad v_{j+1}-\ldots-v_n
\]
(if $j=1$, then $L=\emptyset$, and if $j=n$, then $R=\emptyset$),
then
\[
\bigl\{\{v_1,\ldots,v_n\}\bigr\}\cup L\cup R
\]
is a construction of the path graph
\[
v_1-v_2-\ldots-v_n.
\]

\begin{rem}\label{r6}
If $K$ is a construction of $v_1-\ldots-v_n$, then
${\{v_1,\ldots,v_n\}\in K}$.
\end{rem}

\begin{rem}\label{r3}
Every construction of $v_1-\ldots-v_n$ has exactly $n$ elements.
\end{rem}

The following proposition stems from
\cite[Proposition~6.11]{DP10}.
\begin{prop}\label{p1}
$V$ is a 0-face derived from
\[
\bigl\{\{i_n,\ldots,i_1\},\ldots,\{i_n,i_{n-1}\},\{i_n\}\bigr\}
\]
iff $V$ is a construction of the path graph
\[
\{i_n,\ldots,i_1\}-\ldots-\{i_n,i_{n-1}\}-\{i_n\}.
\]
\end{prop}

From this proposition and the relation between the simplicial
complex $C^M$ and the $(n-1)$-dimensional associahedron mentioned
in Section~\ref{two-fold}, one can extract the following
recurrence relation for the Catalan numbers
\[
c_0=1\quad{\rm and}\quad c_{n+1}=\sum_{i=0}^n c_i c_{n-i}.
\]

For $K\in C$ and $A\in K$, we say that $a\in A$ is
$A$-\emph{superficial} with respect to $K$, when for every $B\in
K$ such that $B\subset A$, we have that $a\not\in B$. The number
of $A$-superficial elements with respect to $K$ is denoted by
$\nu_K(A)$. From the definition of construction and
Proposition~\ref{p1} we have the following.

\begin{rem}\label{r2}
If $K$ is a 0-face, then for every $A\in K$ we have that
$\nu_K(A)=1$.
\end{rem}

Let $E$ be a 1-face of the second type, i.e.\ $E$ is obtained from
a 0-face $V$ by removing its maximal 0-nested set $M$. By
Remark~\ref{r2}, $\nu_V(M)=1$, and hence, $\bigcup E$ is of the
form
\[
\bigl\{\{i_n,\ldots,i_1\},\ldots,\{i_n,\ldots,i_{j-1}\},\{i_n,\ldots,i_{j+1}\},
\ldots,\{i_n\}\bigr\},
\]
which in the limit cases reads
\[
\bigl\{\{i_n,\ldots,i_2\},\ldots,\{i_n\}\bigr\} \quad {\rm or}
\quad \bigl\{\{i_n,\ldots,i_1\},\ldots,\{i_n,i_{n-1}\}\bigr\}.
\]

By repeating the procedure we used for 0-faces and 1-faces of the
first type, starting with $\bigcup E$, the following product (with
two place holders vacant) is obtained
\[
i_0 \stackrel{\{i_n,\ldots,i_1\}}{\cdot} i_1 {\cdot}\ldots {\cdot}
i_{j-2} \stackrel{\{i_n,\ldots,i_{j-1}\}}{\cdot}
\underline{\hspace{1em}} \cdot \underline{\hspace{1em}}
\stackrel{\{i_n,\ldots,i_{j+1}\}}{\cdot} i_{j+1} {\cdot} \ldots
\stackrel{\{i_n\}}{\cdot} i_n,
\]
and a complete bracketing of this product is reconstructed in the
same way as for 0-faces. (Formally, the outermost brackets are
omitted, but we take them for granted.) Note that this product
must contain the following brackets:
\[
(i_0 \stackrel{\{i_n,\ldots,i_1\}}{\cdot} i_1 {\cdot} \ldots
{\cdot} i_{j-2} \stackrel{\{i_n,\ldots,i_{j-1}\}}{\cdot}
\underline{\hspace{1em}}\,) \cdot (\,\underline{\hspace{1em}}
\stackrel{\{i_n,\ldots,i_{j+1}\}}{\cdot} i_{j+1} {\cdot} \ldots
\stackrel{\{i_n\}}{\cdot} i_n).
\]
There are exactly two 0-faces that include a 1-face $E$ of the
second type. They correspond to two possible permutations of
$i_{j-1},i_j$ in the vacant positions of the bracketing
corresponding to~$E$. Hence, every 1-face of the second type
corresponds to a $\sigma$-term.

\begin{exm}\label{e5}
For $X=\{0,1,2,3\}$, and $E=\{\{\{1\}\},\{\{1,0,3\}\}\}$, we have
that $E$ corresponds to the following bracketing
\[
(2 \stackrel{\{1,0,3\}}{\cdot} \underline{\hspace{1em}}\,) \cdot
(\,\underline{\hspace{1em}} \stackrel{\{1\}}{\cdot} 1).
\]
The 0-face $V$ from Example~\ref{e3} and the 0-face
\[
V''=\{\{\{1,3,0\},\{1,3\},\{1\}\},\{\{1\}\},\{\{1,3,0\}\} \},
\]
which corresponds to the complete bracketing
\[
((2 \cdot 0) \cdot (3 \cdot 1)),
\]
are the 0-faces that include $E$. In the picture given at the end
of Section~\ref{realisation}, the 1-face $E$ corresponds to the
edge incident to the dodecagon covering the label $\{1\}$ and
dodecagon covering the label $\{0,1,3\}$.
\end{exm}

As before, for every complete bracketing $P$ of a product with two
vacant positions, one being the rightmost in the left factor of
$P$, and the other being the leftmost in the right factor of $P$,
one can find a unique 1-face of the second type that corresponds
to $P$. This, together with the analogous fact concerning the
1-faces of the first type, entails that the 1-faces of $C$ are in
one-to-one correspondence with the edges of $\mathcal{G}^*_n$.

\subsection{2-faces}

The 2-\emph{faces} of $C$ are obtained from the 1-faces by
removing a single element. There are two main types of 2-faces of
$C$; a 2-face of the \emph{first type} contains a maximal 0-nested
set, and a 2-face of the \emph{second type} does not contain a
maximal 0-nested set.

If $F$ is a 2-face of the first type, and
\[
M=\bigl\{\{i_n,\ldots,i_1\},\ldots,\{i_n,i_{n-1}\},\{i_n\}\bigr\}
\]
is the maximal 0-nested set contained in $F$, then a unique
bracketing of the product $i_0\cdot i_1\cdot\ldots\cdot i_n$, with
two pairs of brackets omitted, is assigned to $F$ in the same way
as a complete bracketing has been assigned to a 0-face. There are
three possible situations:

\vspace{2ex}

(1.1) there is $A\in F$ such that $\nu_F(A)=3$;

\vspace{1ex}

(1.2) there are incomparable $A,B\in F$ such that
$\nu_F(A)=\nu_F(B)=2$;

\vspace{1ex}

(1.3) there are $A, B\in F$ such that $B\subset A$ and
$\nu_F(A)=\nu_F(B)=2$.

\vspace{2ex}

By a straightforward analysis, in the case (1.1), one concludes
that $F$ is included in exactly five 1-faces and five 0-faces and
it corresponds to Mac Lane's pentagon (3). In the cases (1.2) and
(1.3), $F$ is included in exactly four 1-faces (corresponding to
$\alpha$-terms) and four 0-faces. In the case (1.2), $F$
corresponds to a functoriality quadrilateral (1), while in the
case (1.3) it corresponds to a naturality quadrilateral~(4).

\begin{exm}\label{e6}
Let $X=\{0,1,2,3,4,5\}$ and let $M$ be the maximal 0-nested set
\[
\bigl\{\{0,1,2,3,4\},\{0,1,2,3\},\{0,1,2\}, \{0,1\},\{0\}\bigr\}.
\]
The 2-face
$F_1=\{M,\{\{0,1,2,3\},\{0,1,2\},\{0,1\},\{0\}\},\{\{0\}\}\}$,
corresponds to the bracketing $(5\cdot(4\cdot 3\cdot 2\cdot(1\cdot
0)))$.

\begin{center}
\begin{picture}(120,110)

\put(60,100){\makebox(0,0){$5\cdot(4\cdot(3\cdot (2\cdot(1\cdot
0))))$}} \put(-10,70){\makebox(0,0){$5\cdot(4\cdot((3\cdot
2)\cdot(1\cdot 0)))$}}
\put(-10,40){\makebox(0,0){$5\cdot((4\cdot(3\cdot 2))\cdot(1\cdot
0))$}} \put(60,10){\makebox(0,0){$5\cdot(((4\cdot 3)\cdot
2)\cdot(1\cdot 0))$}} \put(130,55){\makebox(0,0){$5\cdot((4\cdot
3)\cdot (2\cdot(1\cdot 0)))$}}

\put(50,90){\line(-3,-1){30}} \put(0,60){\line(0,-1){10}}
\put(20,30){\line(3,-1){30}} \put(70,90){\line(2,-1){40}}
\put(110,40){\line(-2,-1){40}}

\put(-5,88){\makebox(0,0){\scriptsize $5\cdot(4\cdot(3\cdot
2\cdot(1\cdot 0)))$}} \put(-40,55){\makebox(0,0){\scriptsize
$5\cdot(4\cdot(3\cdot 2)\cdot(1\cdot 0))$}}
\put(-5,22){\makebox(0,0){\scriptsize $5\cdot((4\cdot 3\cdot
2)\cdot(1\cdot 0))$}} \put(130,82){\makebox(0,0){\scriptsize
$5\cdot(4\cdot 3\cdot (2\cdot(1\cdot 0)))$}}
\put(130,27){\makebox(0,0){\scriptsize $5\cdot((4\cdot 3)\cdot
2\cdot(1\cdot 0))$}}

\end{picture}
\end{center}

The 2-face
$F_2=\{M,\{\{0,1,2,3,4\},\{0,1,2,3\}\},\{\{0,1\},\{0\}\}\}$,
corresponds to the bracketing $((5\cdot 4\cdot 3)\cdot(2\cdot
1\cdot 0))$ and it is easy to make the corresponding functoriality
quadrilateral.

The 2-face $F_3=\{M,\{\{0,1,2,3\},\{0,1,2\},\{0,1\},\{0\}\},
\{\{0,1\},\{0\}\}\}$ corresponds to the bracketing $(5\cdot
(4\cdot 3\cdot(2\cdot 1\cdot 0)))$ and it is easy to make the
corresponding naturality quadrilateral.

\end{exm}

If $F$ is a 2-face of the second type, then there are two
possibilities
\[
(2.1)\;\; |\bigcup F|=n-1\quad {\rm or}\quad (2.2)\;\; |\bigcup
F|=n-2.
\]
In the case (2.1), a quadrilateral of a form of the diagrams (8)
corresponds to $F$. In the case (2.2), $\bigcup F$ is of the form:
\[
\bigl\{\{i_n,\ldots,i_1\},\ldots,\{i_n,i_{n-1}\},\{i_n\}\bigr\}-
\bigl\{\{i_n,\ldots,i_j\},\{i_n,\ldots,i_k\}\bigr\},
\]
for $1\leq j<k\leq n$, and we distinguish the following
situations:

\vspace{2ex}

(2.2.1) if $k-j>1$, then an octagon of a form of the diagram (9)
corresponds to~$F$;

\vspace{1ex}

(2.2.2) if $k-j=1$, then a dodecagon of a form of the diagram (10)
corresponds to~$F$.

\begin{exm}\label{e7}
For $X=\{0,1,2,3,4,5\}$, the 2-face
\[
F_4=\Bigl\{\bigl\{\{0,1,2,3,4\}\bigr\},\;
\bigl\{\{0,1,2\},\{0,1\},\{0\}\bigr\},\;
\bigl\{\{0\}\bigr\}\Bigr\}
\]
describes the bracketing
$(5\cdot\underline{\hspace{1em}}\,)\cdot(\,\underline{\hspace{1em}}\cdot
2\cdot(1\cdot 0))$.

\begin{center}
\begin{picture}(120,80)(0,-10)

\put(-25,50){\makebox(0,0){$(5\cdot 4)\cdot(3\cdot(2\cdot(1\cdot
0)))$}} \put(-25,10){\makebox(0,0){$(5\cdot
3)\cdot(4\cdot(2\cdot(1\cdot 0)))$}}
\put(145,50){\makebox(0,0){$(5\cdot 4)\cdot((3\cdot 2)\cdot(1\cdot
0))$}} \put(145,10){\makebox(0,0){$(5\cdot 3)\cdot((4\cdot
2)\cdot(1\cdot 0))$}}

\put(-20,40){\line(0,-1){20}} \put(140,40){\line(0,-1){20}}
\put(30,50){\line(1,0){60}} \put(30,10){\line(1,0){60}}

\put(-65,30){\makebox(0,0){\scriptsize $(5\cdot
\underline{\hspace{.5em}}\,)\cdot(\,\underline{\hspace{.5em}}\cdot(2\cdot(1\cdot
0)))$}} \put(185,30){\makebox(0,0){\scriptsize $(5\cdot
\underline{\hspace{.5em}}\,)\cdot((\,\underline{\hspace{.5em}}\cdot
2)\cdot(1\cdot 0))$}} \put(60,60){\makebox(0,0){\scriptsize
$(5\cdot 4)\cdot(3\cdot 2\cdot (1\cdot 0))$}}
\put(60,0){\makebox(0,0){\scriptsize $(5\cdot 3)\cdot(4\cdot
2\cdot (1\cdot 0))$}}

\end{picture}
\end{center}

The 2-face $F_5=\{\{\{0,1,2,3,4\}\},\{\{0,1,2\}\},\{\{0\}\}\}$,
describes the bracketing
\[
(5\cdot \underline{\hspace{1em}}\,) \cdot
(\,\underline{\hspace{1em}} \cdot \underline{\hspace{1em}}\,)\cdot
(\,\underline{\hspace{1em}}\cdot 0),
\]
where the first two vacant
positions are reserved for 3 and 4, while the last two are
reserved for 1 and 2. This face corresponds to an octagon of a
form of the diagram (9).

The 2-face $F_6=\{\{\{0,1,2,3,4\}\},\{\{0,1\},\{0\}\},\{\{0\}\}\}$
describes the bracketing
\[
(5\cdot \underline{\hspace{1em}}\,)\cdot
\underline{\hspace{1em}}\cdot (\,\underline{\hspace{1em}}\cdot
(1\cdot 0))
\]
and it corresponds to a dodecagon of a form of the diagram (10).

\end{exm}

\section{Simple permutoassociahedra}\label{realisation}

In this section we present a polytope $\mathbf{PA}_n$, the
$n$-dimensional \emph{simple permutoassociahedron}, and give the
positive answer to Question~\ref{q1}. We start with the following
notation. For $1\leq k\leq k+l\leq n$, let
\[
\kappa(k,l)=\frac{3^{k+l+1}-3^{l+1}}{2}+\frac{3^k-3k}{3^n-n-1}.
\]
For an element
\[
\beta=\bigl\{\{i_{k+l},\ldots,i_k,\ldots,i_1\},\ldots,\{i_{k+l},\ldots,i_k,i_{k-1}\},\{i_{k+l},\ldots,i_k\}
\bigr\},
\]
of $\mathcal{B}_1$, let $\beta^=$ be the equation
(\emph{hyperplane} in $\mathbf{R}^{n+1}$)
\[
x_{i_1}+2x_{i_2}+\ldots+k(x_{i_k}+\ldots+x_{i_{k+l}})=\kappa(k,l),
\]
and let the \emph{halfspace} $\beta^{\geq}$ and the \emph{open
halfspace} $\beta^>$ be defined as $\beta^=$, save that ``$=$'' is
replaced by ``$\geq$'' and ``$>$'', respectively.

\begin{rem}\label{r4}
We have that $\kappa(k,l)=3^{l+k}+\ldots+3^{l+1}+\varepsilon(k)$,
for
\[
\varepsilon(k)=\frac{3^k-3k}{3^n-n-1}<1.
\]
\end{rem}

For $\pi$ being the hyperplane $x_0+\ldots +x_n=3^{n+1}$ in
$\mathbf{R}^{n+1}$, let
\[
\mathbf{PA}_n=(\bigcap \{\beta^{\geq}\mid \beta\in
\mathcal{B}_1\})\cap \pi.
\]
The rest of this section is devoted to a proof of the following
result.
\begin{thm}\label{t1}
$\mathbf{PA}_n$ is a simple $n$-dimensional polytope that realises
$C$.
\end{thm}

For every $0\leq i\leq n$, we have that $\{\{i\}\}\in
\mathcal{B}_1$. Hence, $\mathbf{PA}_n$ is a subset of the
$n$-simplex
\[
x_0+\ldots+x_n=3^{n+1},\quad\quad x_i\geq 3,\quad 0\leq i\leq n,
\]
and therefore a polytope, and not just a polyhedron.

For $\mathcal{B}_1^M$ being defined as $\mathcal{B}_1\cap
\mathcal{P}(M)$, we have the following.

\begin{lem}\label{l4}
If $\beta,\gamma\in \mathcal{B}_1$ are such that $\beta^{=}\cap
\gamma^{=}\cap \mathbf{PA}_n\neq\emptyset$, then there is a
maximal 0-nested set $M$ such that
$\beta,\gamma\in\mathcal{B}_1^M$.
\end{lem}
\begin{proof}
Suppose that
\begin{itemize}
\item[$(\ast)$] there is no maximal 0-nested set $M$ such that
$\beta,\gamma\in\mathcal{B}_1^M$.
\end{itemize}
Let $\beta^{=}$ be the hyperplane
\[
y_1+2y_2+\ldots+k(y_k+\ldots+y_{k+l})=\kappa(k,l),
\]
and let $\gamma^{=}$ be the hyperplane
\[
z_1+2z_2+\ldots+m(z_m+\ldots+z_{m+p})=\kappa(m,p),
\]
where $Y=\{y_1,\ldots,y_{k+l}\}$ and $Z=\{z_1,\ldots,z_{m+p}\}$
are subsets of the set of variables $\{x_0,\ldots,x_n\}$. Let
$\Sigma$ be the following sum:
\[
y_1+2y_2+\ldots+k(y_k+\ldots+y_{k+l})+
z_1+2z_2+\ldots+m(z_m+\ldots+z_{m+p}).
\]
We show that $\Sigma>\kappa(k,l)+\kappa(m,p)$, for every point
$a(x_0,x_1,\ldots,x_n)\in \mathbf{PA}_n$, which entails that
$\mathbf{PA}_n$ and the hyperplane
$\Sigma=\kappa(k,l)+\kappa(m,p)$ are disjoint. Since
$\beta^{=}\cap \gamma^{=}$ is included in this hyperplane, we
conclude that $a\not\in \beta^{=}\cap \gamma^{=}$. In order to
prove this, we use the following fact several times: if
$i_1,\ldots,i_m$ are mutually distinct elements of
$\{0,\ldots,n\}$, then $\{\{i_1,\ldots,i_m\}\}\in \mathcal{B}_1$,
and for $a(x_0,x_1,\ldots,x_n)\in \mathbf{PA}_n$, we have that
$x_{i_1}+\ldots+x_{i_m}\geq 3^m$.

\vspace{2ex}

(1) If $Y$ and $Z$ are incomparable and $Y\cup
Z=\{u_1,\ldots,u_q\}$, then we have that $q>k+l,m+p$ and
\[
\Sigma\geq u_1+\ldots+ u_q\geq 3^q\geq \frac{3^{k+l+1} +
3^{m+p+1}}{2}>\kappa(k,l)+\kappa(m,p).
\]

\vspace{2ex}

(2) If $Z\subset Y$, then, by $(\ast)$, we conclude that either

\vspace{1ex}

(2.1) $z_1=y_i$ for some $i\leq k-2$ and $
\{z_2,\ldots,z_{m+p}\}\subseteq\{y_k,\ldots,y_{k+l}\}$, or

\vspace{1ex}

(2.2) $z_1=y_{k-1}$ and
$\{z_2,\ldots,z_{m+p}\}\subset\{y_k,\ldots,y_{k+l}\}$, or

\vspace{1ex}

(2.3) $\{z_2,\ldots,z_{m+p}\}\not\subseteq\{y_k,\ldots,y_{k+l}\}$.

\vspace{1ex}

In the case (2.1), we have that $k+l-i\geq p+m$, and
\begin{tabbing}
\hspace{2em}$\Sigma$ \= $\geq (y_1+\ldots+y_{l+k})+\ldots+
(y_{i-1}+\ldots+y_{l+k})+ 2(y_i+\ldots+y_{l+k})$
\\[1ex]
\> $\geq 3^{l+k}+\ldots+3^{l+k-i+2}+2\cdot 3^{l+k-i+1}$
\\[1ex]
\> $> 3^{l+k}+\ldots+3^{l+k-i+2}+ 3^{l+k-i+1}+
2(3^{l+k-i}+\ldots+1)$
\\[1ex]
\> $\geq 3^{l+k}+\ldots+3^{l+1}+1+ 3^{p+m}+ \ldots+3^{p+1}+1$,
since $k+l-i\geq p+m$
\\[1ex]
\> $> \kappa(k,l)+\kappa(m,p)$.
\end{tabbing}

In the case (2.2), we have that $l+1\geq p+m$, and we repeat the
above calculation with $i$ replaced by $k-1$.

In the case (2.3), let $i$ be the smallest element of
$\{1,\ldots,k-1\}$ such that $y_i=z_j$ for some
$j\in\{2,\ldots,m+p\}$. We have that $k\!+\!l\!-\!i\!+\!1\geq
p\!+\!m$ and
\begin{tabbing}
\hspace{2em}$\Sigma$ \= $\geq (y_1+\ldots+y_{l+k})+\ldots+
(y_{i-1}+\ldots+y_{l+k})+ 3(y_i+\ldots+y_{l+k})$
\\[1ex]
\> $\geq 3^{l+k}+\ldots+3^{l+k-i+2}+3\cdot 3^{l+k-i+1}$
\\[1ex]
\> $> 3^{l+k}+\ldots+3^{l+k-i+2}+ 2(3^{l+k-i+1}+\ldots+1)$
\\[1ex]
\> $\geq 3^{l+k}+\ldots+3^{l+1}+1+ 3^{p+m}+ \ldots+3^{p+1}+1$,
since $k\!+\!l\!-\!i\!+\!1\geq p\!+\!m$
\\[1ex]
\> $> \kappa(k,l)+\kappa(m,p)$.
\end{tabbing}

\vspace{2ex}

(3) If $Z=Y$, then, by $(\ast)$, there is $i\in\{1,\ldots,k-1\}$
such that for some $r\geq 1$, $\gamma^{=}$ is of the form
\[
y_1+\ldots+(i-1)y_{i-1}+(i+r)y_i+\ldots=\kappa(m,p),
\]
where the coefficient of every $y_j$, $j>i$, hidden in
``$\ldots$'', is greater or equal to $i$. Hence,
\begin{tabbing}
\hspace{2em}$\Sigma$ \= $\geq 2(y_1+\ldots+y_{l+k})+\ldots+
2(y_{i-1}+\ldots+y_{l+k})+ 3(y_i+\ldots+y_{l+k})$
\\[1ex]
\> $\geq 2(3^{l+k}+\ldots+3^{l+k-i+1})+ 3^{l+k-i+1}$
\\[1ex]
\> $> 2(3^{l+k}+\ldots+1)$, since
$3^{l+k-i+1}-1=2(3^{k+l-i}+\ldots+1)$
\\[1ex]
\> $> \kappa(k,l)+\kappa(m,p)$.
\end{tabbing}
\end{proof}

The following lemma is the basis of the inductive proof of
\cite[Lemma~9.1]{DP10}.

\begin{lem}\label{l3}
Let $Y$ and $Z$ be two incomparable subsets of $\{0,\ldots,n\}$.
If for every $i\in Y\cup Z$ we have that $x_i\geq 0$, and
$\sum_{i\in Y}x_i\leq 3^{|Y|}$, $\sum_{i\in Z}x_i\leq 3^{|Z|}$,
then \[\sum_{i\in Y\cup Z}x_i< 3^{|Y\cup Z|}.\]
\end{lem}

In the sequel, we rely on the following affine transformation that
normalises the equations of hyperplanes localised at
$\mathcal{B}_1^M$ for
\[
M=\bigl\{\{n,\ldots,1\},\ldots,\{n,n-1\},\{n\}\bigr\}.
\]

Let $p$ be the orthogonal projection from the hyperplane $\pi$ to
the hyperplane ${\pi^0: x_0=0}$, and let $\mathbf{PA}_n^0$ be the
$p$-image of $\mathbf{PA}_n$. Since $p$ is an affine bijection,
the polytopes $\mathbf{PA}_n$ and $\mathbf{PA}_n^0$ are
combinatorially equivalent and $p$ maps a face of $\mathbf{PA}_n$
to the corresponding face of $\mathbf{PA}_n^0$.

Consider the $n\times n$ matrices $L$ and $L^{-1}$
\[
\frac{1}{3^n\!-\!n\!-\!1}\left( \begin{array}{ccccc} 1 & -1 & 0 &
\ldots & 0
\\
0 & 1 & -1 & \ldots & 0
\\
\vdots & \vdots & \vdots & \vdots & \vdots
\\
0 & 0 & 0 & \ldots & 1
\end{array}\right),
\quad (3^n\!-\!n\!-\!1)\left( \begin{array}{ccccc} 1 & 1 & 1 &
\ldots & 1
\\
0 & 1 & 1 & \ldots & 1
\\
\vdots & \vdots & \vdots & \vdots & \vdots
\\
0 & 0 & 0 & \ldots & 1
\end{array}\right)
\]
and the vector
\[
\mathbf{w}=\left(\begin{array}{c} 3^n-3^{n-1}
\\
3^{n-1}-3^{n-2}
\\
\vdots
\\
6
\\
3
\end{array} \right)
\]
Let $h:\mathbf{R}^n\str \pi^0$ be the affine bijection mapping
$(x_1',\ldots,x_n')\in \mathbf{R}^n$ into the point
$(x_0,\ldots,x_n)\in\pi^0$, such that $x_0=0$ and
\[
\left(\begin{array}{c} x_1
\\
\vdots
\\
x_n
\end{array} \right)= \mathbf{w} + L \left(\begin{array}{c}
x_1'-3
\\
\vdots
\\
x_n'-3
\end{array} \right).
\]
Let $u=h^{-1}\circ p$ and let $\mathbf{PA}_n'$ be the polytope
obtained as the $u$-image of $\mathbf{PA}_n$. The polytopes
$\mathbf{PA}_n$ and $\mathbf{PA}_n'$ are again combinatorially
equivalent.

What are the hyperplanes of $\mathbf{R}^n$ that correspond by $u$
to $\beta^=\cap\pi$, where $\beta\in\mathcal{B}_1^M$? Every such
$\beta$ is of the form
\[
\bigl\{\{n,n-1,\ldots,n-l,\ldots,n-l-k+1\},\ldots,\{n,n-1,\ldots,n-l\}\bigr\},
\]
and let us define the $u$-\emph{correspondent} of $\beta$ to be
the set $Y=\{n-l-k+1,\ldots,n-l\}$.

The hyperplane $\beta^{=}$ in $\mathbf{R}^{n+1}$ is of the form
\[
x_{n-l-k+1}+2x_{n-l-k+2}+\ldots+(k-1)x_{n-l-1}+k(x_{n-l}+\ldots+x_n)=\kappa(k,l),
\]
and let us denote the left-hand side of this equation by $\Xi$.
For $Y$ being the $u$-correspondent of $\beta$, we define the
hyperplane $Y^{=}$ in $\mathbf{R}^n$ by the equation
\[
\sum_{i\in Y} x'_i=3^{|Y|},\quad{\rm i.e.}\quad
x_{n-l-k+1}'+\ldots+x_{n-l}'=3^k.
\]

We show that $\beta^{=}\cap\pi^0$ is the $h$-image of $Y^{=}$. By
the definition of $h$, we have that
\[
\begin{array}{rl}
x_{n-l-k+1}\!\!\!\! & =2\cdot
3^{l+k-1}+\frac{1}{3^n-n-1}(x_{n-l-k+1}'-x_{n-l-k+2}')
\\
\vdots & \quad\quad\vdots
\\
x_{n-l}\!\!\!\! & =2\cdot
3^{l}+\frac{1}{3^n-n-1}(x_{n-l}'-x_{n-l+1}')
\\
\vdots & \quad\quad\vdots
\\
x_{n}\!\!\!\! & =3+\frac{1}{3^n-n-1}(x_{n}'-3),\; {\rm and}
\end{array}
\]
\[
\begin{array}{rl}
\Xi\!\!\!\! & =2\cdot 3^{l+k-1}+ 2\cdot 2 \cdot
3^{l+k-2}+\ldots+(k-1) 2\cdot 3^{l+1}+k(2\cdot 3^l+\ldots +2\cdot
3+3)
\\[1ex]
 & \;\;\;+ \frac{1}{3^n-n-1}(x_{n-l-k+1}'+\ldots+x_{n-l}'-3 k)
\\[2ex]
 & =3^{l+k}+3^{l+k-1}+\ldots+3^{l+1}+ \frac{1}{3^n-n-1}(x_{n-l-k+1}'+\ldots+x_{n-l}'-3
 k).
\end{array}
\]
Hence,
\[
\begin{array}{rcl}
\Xi=\kappa(k,l) & \Leftrightarrow &
\frac{1}{3^n-n-1}(x_{n-l-k+1}'+\ldots+x_{n-l}'-3k)=
\frac{3^k-3k}{3^n-n-1}
\\[1ex]
 & \Leftrightarrow & x_{n-l-k+1}'+\ldots+x_{n-l}'= 3^k,
\end{array}
\]
which means that $\beta^{=}\cap\pi^0$ is the $h$-image of $Y^{=}$,
i.e.\ the hyperplane $Y^=$ of $\mathbf{R}^n$ corresponds by $u$ to
$\beta^=\cap\pi$. The same holds when the superscript ``$=$'' is
replaced by ``$\geq$'' in which case the equation $Y^=$ is
replaced by the appropriate inequality.

Hence the $u$-correspondence between the members of
$\mathcal{B}_1^M$ and the members of the set
\[
\mathcal{B}=\bigl\{\{n-l,\ldots,n-l-k+1\}\mid 1\leq k\leq k+l\leq
n\bigr\},
\]
which is a bijection, is lifted to a bijection between a set of
hyperplanes (halfspaces) tied to $\mathbf{PA}_n$ and a set of
hyperplanes (halfspaces) tied to $\mathbf{PA}_n'$.

Note that $\mathcal{B}$ is the set of all non-empty connected sets
of vertices of the path graph $1-2-\ldots-n$ and it is a building
set of $\mathcal{P}(\{1,\ldots,n\})$. It is easy to check that the
following implications and its converses hold. If two members of
$\mathcal{B}_1^M$ are incomparable, then the $u$-correspondents
are incomparable. If the union of two members of $\mathcal{B}_1^M$
is in $\mathcal{B}_1^M$, then the union of the $u$-correspondents
is in $\mathcal{B}$. If $N\subseteq \mathcal{B}_1^M$ is 1-nested,
then the set of $u$-correspondents of members of $N$ is a nested
set with respect to $\mathcal{B}$. If $N\subseteq \mathcal{B}_1^M$
is a construction of the path graph
\[
\{n,\ldots,1\}-\ldots-\{n,n-1\}-\{n\},
\]
then the set of $u$-correspondents of members of $N$ is a
construction of the path graph $1-2-\ldots-n$.

\begin{lem}\label{l5}
If $\beta,\gamma\in \mathcal{B}_1$ are such that $\beta^{=}\cap
\gamma^{=}\cap \mathbf{PA}_n\neq\emptyset$, and $\beta$, $\gamma$
are incomparable, then $\beta\cup\gamma\in C_1-\mathcal{B}_1$.
\end{lem}

\begin{proof}
From Lemma~\ref{l4}, we conclude that there is a maximal 0-nested
set $M$ such that $\beta,\gamma\in \mathcal{B}_1^M$. By symmetry,
we may assume that
\[
M=\bigl\{\{n,\ldots,1\},\ldots,\{n,n-1\},\{n\}\bigr\}.
\]
(In that case $x_0$ occurs neither in $\beta^{=}$ nor in
$\gamma^{=}$, since $\beta,\gamma\in \mathcal{B}_1^M$.)

Let $Y,Z\in\mathcal{B}$ be the $u$-correspondents of $\beta$ and
$\gamma$, respectively. Since $\beta$ and $\gamma$ are
incomparable, $Y$ and $Z$ are incomparable too. Let $a\in
\beta^{=}\cap \gamma^{=}\cap \mathbf{PA}_n$ and let $a'=u(a)$.
Then the coordinates $(x_1',\ldots,x_n')$ of $a'$ satisfy the
equations $Y^=$ and $Z^=$, and hence, the inequalities $Y^\leq$
and $Z^\leq$.

This means that $Y$, $Z$ and the coordinates of $a'$, satisfy the
conditions of Lemma~\ref{l3} (the coordinates of $a'$ are positive
since $x_i'\geq 3$ holds for every $1\leq i\leq n$). Hence, we
have that $Y\cup Z\not\in \mathcal{B}$, otherwise, the equation
$(Y\cup Z)^{\geq}$, i.e.\ $\sum_{i\in Y\cup Z} x_i'\geq 3^{|Y\cup
Z|}$ would hold, which contradicts Lemma~\ref{l3}. From this we
conclude that $\beta\cup\gamma\not\in \mathcal{B}_1$, and since
both $\beta$ and $\gamma$ are in $\mathcal{B}_1^M$, we have that
$\beta\cup\gamma\in C_1-\mathcal{B}_1$.
\end{proof}

As a consequence of Lemma~\ref{l5} and Proposition~\ref{p2} we
have the following.

\begin{cor}\label{c2}
If $\beta_1,\ldots,\beta_k\in \mathcal{B}_1$ are such that
$\beta_1^{=}\cap\ldots\cap \beta_k^{=}\cap
\mathbf{PA}_n\neq\emptyset$, then $\{\beta_1,\ldots,\beta_k\}$ is
1-nested.
\end{cor}

\begin{prop}\label{p3}
For every vertex $v$ of $\mathbf{PA}_n$, there is a 0-face $V$ of
$C$ such that
\[
\{v\}=(\bigcap \{\beta^{=}\mid \beta\in V\})\cap \pi,
\]
and for every $\gamma\in \mathcal{B}_1-V$, $v\not\in\gamma^{=}$.
\end{prop}
\begin{proof}
Let $v$ be a vertex of $\mathbf{PA}_n\subseteq \mathbf{R}^{n+1}$.
Since the dimension of this space is $n+1$, one concludes that
there are $\beta_1,\ldots,\beta_n\in \mathcal{B}_1$ such that
\[
\{v\}=\beta_1^{=}\cap\ldots\cap \beta_n^{=} \cap \pi.
\]
From Corollary~\ref{c2}, it follows that
$V=\{\beta_1,\ldots,\beta_n\}$ is a 1-nested set, and since it has
$n$ elements, by Remark~\ref{r1}, it is a 0-face of $C$. It
remains to prove that for every $\gamma\in \mathcal{B}_1 -V$, we
have that $v\not\in\gamma^{=}$. This is straightforward, since by
maximality of $V$, we have that $V\cup \{\gamma\}$ is not
1-nested, and by Corollary~\ref{c2}, the vertex $v$ cannot be
in~$\gamma^{=}$.
\end{proof}

\begin{cor}\label{c1}
$\mathbf{PA}_n$ is a simple $n$-dimensional polytope.
\end{cor}

\begin{proof}
This is obvious after Proposition~\ref{p3}, or we may apply
\cite[Proposition~9.7]{DP10}, since every 0-face has exactly $n$
elements.
\end{proof}

\begin{prop}\label{p4}
For every 0-face $V$ of $C$, there is a vertex $v$ of
$\mathbf{PA}_n$ such that
\[
\{v\}=(\bigcap \{\beta^{=}\mid \beta\in V\})\cap \pi.
\]
\end{prop}

\begin{proof}
We show that for every 0-face $V$ of $C$, $(\bigcap
\{\beta^{=}\mid \beta\in V\})\cap \pi$ is a singleton $\{v\}$ such
that $v\in \gamma^>$, for every $\gamma\in \mathcal{B}_1-V$, from
which the proposition follows. Again, by symmetry, we may assume
that $V=\{\beta_1,\ldots,\beta_n\}\subseteq \mathcal{B}_1^M$, for
\[
M=\bigl\{\{n,\ldots,1\},\ldots,\{n,n-1\},\{n\}\bigr\}.
\]

Let $X_i$, for every $i\in\{1,\ldots,n\}$, be the
$u$-correspondent of $\beta_i\in V$, for $u$ being the affine
transformation introduced after Lemma~\ref{l3}. Since $V$ is a
construction of
\[
\{n,\ldots,1\}-\ldots-\{n,n-1\}-\{n\},
\]
we have that $N=\{X_1,\ldots,X_n\}$ is a construction of the path
graph $1-2-\ldots-n$.

By induction on $n\geq 1$, we show that $X_1^{=}\cap\ldots\cap
X_n^{=}$ is a singleton. If $n=1$, then $X_1^{=}$ is $x_1'=3$, and
we are done. If $n>1$, let $L$ and $R$ be the constructions of the
path graphs
\[
1-\ldots-(j-1)\quad {\rm and}\quad (j+1)-\ldots-n,
\]
such that $N=\{\{1,\ldots,n\}\}\cup L\cup R$. By the induction
hypothesis, we have that $\bigcap\{\beta_i^{=}\mid \beta_i\in L\}$
and $\bigcap\{\beta_i^{=}\mid \beta_i\in R\}$ are singletons in
the corresponding $(j-1)$-dimensional and $(n-j)$-dimensional
spaces. Hence, the coordinates
$x_1',\ldots,x_{j-1}',x_{j+1}',\ldots,x_n'$ of a point belonging
to the above intersection are determined, and it remains to
determine the coordinate $x_j'$. It has a unique value evaluated
from the equation $x_1'+\ldots+x_{j-1}'+x_j'+x_{j+1}'+\ldots+
x_n'=3^n$ and the values of
$x_1',\ldots,x_{j-1}',x_{j+1}',\ldots,x_n'$.

From this we conclude that
\[
(\bigcap \{\beta^{=}\mid \beta\in V\})\cap \pi
\]
is a singleton $\{v\}$. It remains to show that $v\in \gamma^>$,
for every $\gamma\in \mathcal{B}_1-V$.

Note that in the above calculation of coordinates of $u(v)$, for
the $j$th coordinate we have
\[
x_j'=3^n-(3^{j-1}+3^{n-j})>3^{n-1},
\]
since $x_1'+\ldots+x_{j-1}'=3^{j-1}$ and
$x_{j+1}'+\ldots+x_n'=3^{n-j}$. Note also that
$j\in\{1,\ldots,n\}$ was the $\{1,\ldots,n\}$-superficial element
with respect to $N$. In the same way, we may conclude that if $m$
is the $X_i$-superficial with respect to $N$, then the $m$th
coordinate of $u(v)$ satisfies $x_m'>3^{|X_i|-1}$.

For $\gamma\in \mathcal{B}_1^M -V$, let $Y=\{n-l,\ldots,n-l-k+1\}$
be its $u$-correspondent in $\mathcal{B}$. Since $Y\in
\mathcal{B}-N$, it is easy to prove, by induction on $n$, that
there exists $X_i\in N$ such that $Y\subset X_i$ and $Y$ contains
the $X_i$-superficial element with respect to $N$ (cf.\
\cite[Lemma~6.14]{DP10}). Hence, there is
$m\in\{n-l,\ldots,n-l-k+1\}$, which is the $X_i$-superficial, and
$k<|X_i|$. From the preceding paragraph, we conclude that the
$m$th coordinate of $u(v)$ satisfies $x_m'>3^{|X_i|-1}$, and since
all the other coordinates are positive, we have that $u(v)\in
Y^>$, i.e.\
\[
x_{n-l-k+1}'+\ldots+x_{n-l}'>3^{|X_i|-1}\geq 3^k,
\]
hence, $v\in \gamma^>$.

For every other $\gamma\in \mathcal{B}_1 -V$, i.e.\ when
$\gamma\in \mathcal{B}_1-\mathcal{B}_1^M$, consider the
permutohedron $P_n$ obtained as the intersection of the hyperplane
$\pi$ and the halfspaces of the form
\[
x_{i_1}+\ldots+x_{i_k}\geq 3^k,
\]
where $i_1,\ldots,i_k$ are mutually distinct elements of
$\{0,\ldots,n\}$. This permutohedron realises the simplicial
complex $C_1$. The intersection of $P_n$ and the hyperplane
\[
x_1+2x_2+\ldots+nx_n=\kappa(n,0)
\]
is an $(n-1)$-simplex. It is the convex hull of the set of points
$\{a_1,\ldots,a_n\}\subseteq \mathbf{R}^{n+1}$ (note that
$\varepsilon(n)=\frac{3^n-3n}{3^n-n-1}$):
\[
\begin{array}{l}
a_1(3^{n+1}-3^n-\varepsilon(n), 3^{n}-3^{n-1}+\varepsilon(n),
3^{n-1}-3^{n-2},\ldots,3),
\\[1ex]
a_2(3^{n+1}-3^n, 3^{n}-3^{n-1}-\varepsilon(n),
3^{n-1}-3^{n-2}+\varepsilon(n),\ldots,3),
\\[1ex]
\quad\vdots
\\[1ex]
a_n(3^{n+1}-3^n,\ldots, 3^3-3^2, 3^2-3-\varepsilon(n),
3+\varepsilon(n)).
\end{array}
\]
(The coordinates of $a_i$ are obtained as the solution of the
system
\[
\begin{array}{c r l}
 & x_1+2x_2+\ldots+n x_n & =\kappa(n,0)
\\[1ex]
 & x_0+x_1+\;\;x_2+\ldots+\;\;x_n & =3^{n+1}
\\[1ex]
 (1) & x_1+\;\; x_2+\ldots +\;\; x_n & =3^n
\\[1ex]
 (2) & x_2+\ldots + \;\;x_n & =3^{n-1}
\\[1ex]
 \vdots & &
\\[1ex]
 (n) & x_n & =3,
\end{array}
\]
with the equation $(i)$ omitted.)

Since $\gamma\in \mathcal{B}_1-\mathcal{B}_1^M$, the hyperplane
$\gamma^{=}$ is of the form
\[
x_{i_1}+2
x_{i_2}+\ldots+k(x_{i_k}+\ldots+x_{i_{k+l}})=\kappa(k,l),
\]
where $i_k<i_{k+1}<\ldots<i_{k+l}$ and
\begin{itemize}
\item[$(\ast\ast)$] $i_1,i_2,\ldots, i_{k+l}$ are not consecutive
members of the sequence $1,2,\ldots,n$.
\end{itemize}
Denote the left-hand side of the above equation by $\Xi$.

We show that for every $i\in\{1,\ldots,n\}$, the coordinates
$(x_0,\ldots,x_n)$ of $a_i$ satisfy $\gamma^>$. From this, since
$v$ is in the convex hull of $\{a_1,\ldots, a_n\}$, one obtains
that $v\in\gamma^>$.

\vspace{2ex}

(1) If for some $0\leq m\leq n-l-k$, $x_m$ occurs in $\Xi$, then
\[
\begin{array}{rl}
\Xi\!\!\!\! & > x_m > 3^{n+1-m} - 3^{n-m}-1= 2\cdot 3^{n-m} -1=
2(3^{n-m}-1)+1
\\[2ex]
 & > \frac{3}{2}(3^{n-m}-1)+1= \frac{3^{n-m+1}-3}{2}+1 \geq
 \frac{3^{l+k+1}-3}{2}+1
\\[2ex]
 & > \kappa(k,l).
\end{array}
\]

(2) If $i_1,\ldots,i_{l+k}\in \{n-l-k+1,\ldots,n\}$, then  $\Xi$
is of the form
\[
a_1x_{n-l-k+1}+ a_2 x_{n-l-k+2}+\ldots +a_{l+k}x_n,
\]
where $a_m\geq 0$, for every $0\leq m\leq l+k$. Since $(\ast\ast)$
holds, there is $j\in\{1,\ldots,k-1\}$ such that $a_j>j$ and for
every $m\in\{1,\ldots,j-1\}$, we have that $a_m=m$. Therefore,
\[
\begin{array}{rl}
\Xi\!\!\!\! & \geq x_{n-l-k+1}+\ldots+(j-1) x_{n-l-k+j-1} +
j(x_{n-l-k+j}+\ldots+ x_n)+x_{n-l-k+j}
\\[2ex]
 & = (x_{n-l-k+1}+\ldots+ x_n) +\ldots
+ (x_{n-l-k+j}+\ldots+x_n)+x_{n-l-k+j}.
\end{array}
\]

For $i$ such that $n-l-k+1\leq i\leq n-l-k+j$, the coordinates
$(x_0,\ldots,x_n)$ of the point $a_i$ satisfy the inequality
\[
(x_{n-l-k+1}+\ldots+ x_n) +\ldots + (x_{n-l-k+j}+\ldots+x_n)>
3^{l+k}+\ldots+3^{l+k-j+1}.
\]
For the other values of $i$, the above inequality converts into
equality. Also the $(n-l-k+j)$th coordinate $x_{n-l-k+j}$ of $a_i$
is greater than $3^{l+k-j+1}-3^{l+k-j}-1$. Hence,
\[
\begin{array}{rl}
\Xi\!\!\!\! & > 3^{l+k-j+1}{\displaystyle \frac{3^j-1}{2}} +
3^{l+k-j+1} -3^{l+k-j}-1
\\[2ex]
 & = {\displaystyle \frac{3^{l+k+1}-3^{l+k-j+1}+
 2\cdot 3^{l+k-j+1} -2\cdot 3^{l+k-j}}{2}}-1
 \\[2ex]
 & = {\displaystyle \frac{3^{l+k+1}+ 3^{l+k-j}}{2}}-1
 \\[2ex]
 & \geq {\displaystyle \frac{3^{l+k+1}+ 3^{l+1}}{2}}-1,\quad {\rm
 since\;} k\geq j+1
 \\[2ex]
 & > {\displaystyle \frac{3^{l+k+1}- 3^{l+1}}{2}}+1 > \kappa(k,l).
\end{array}
\]
\end{proof}

\begin{proof}[Proof of Theorem~\ref{t1}]
From Propositions~\ref{p3}, \ref{p4} and Corollary~\ref{c1}, we
conclude that the face lattice of $\mathbf{PA}_n$, with bottom
removed, is given by
\[
\Bigl(\bigcup\bigl\{\mathcal{P}(\{\beta^{=}\mid \beta\in V\})\mid
V\;\mbox{\rm is a 0-face of }C\bigr\},\supseteq \Bigr),
\]
from which the theorem follows.
\end{proof}

\begin{figure}[h!h!h!]
\includegraphics[width=1.0\textwidth]{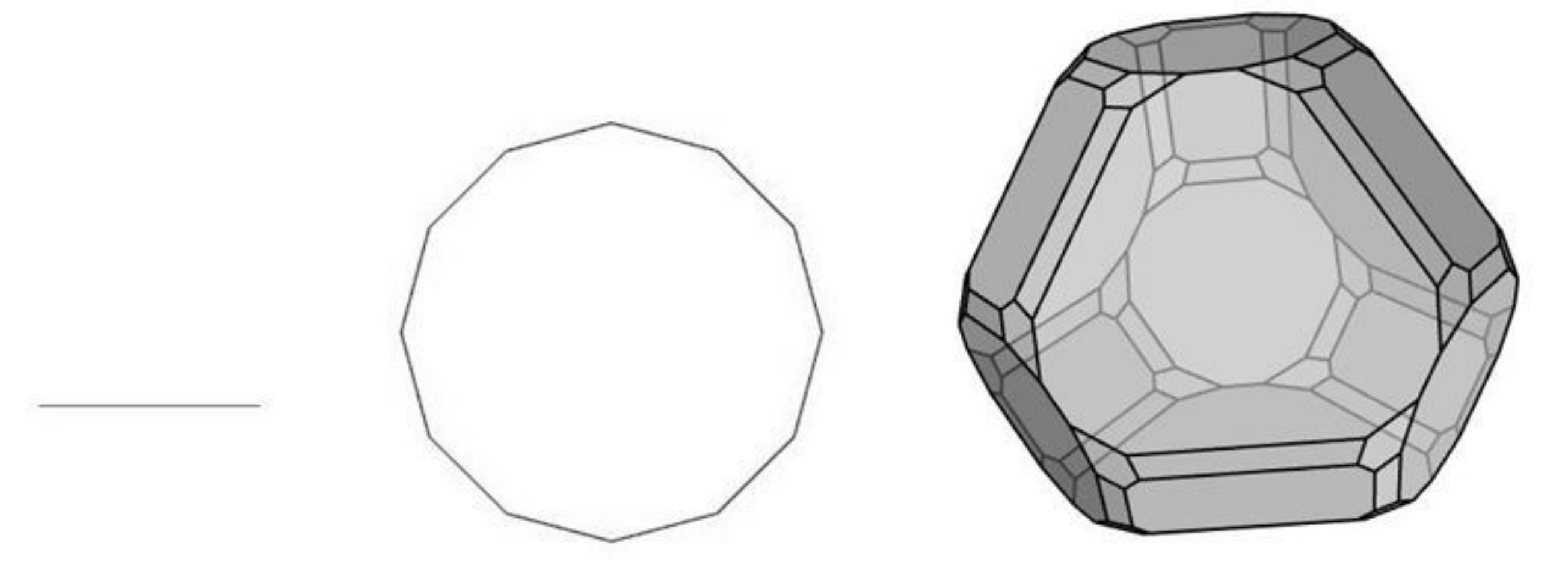}
\caption{$\mathbf{PA}_1$, $\mathbf{PA}_2$ and $\mathbf{PA}_3$}
\label{s:pasn}
\end{figure}

Figure \ref{s:pasn} illustrates the $n$-dimensional simple
permutoassociahedra for $n=1,2$ and~$3$. We conclude the paper
with a picture that indicates the connection between the
3-dimensional simple permutoassociahedron and the 3-dimensional
permutohedron. The facets of the permutohedron are labelled by the
elements of
\[
\mathcal{B}_0=\mathcal{P}(\{0,1,2,3\})-\{\{0,1,2,3\},
\emptyset\}.
\]
As it has been already noted, it is more intuitive to transform
every such $A\in\mathcal{B}_0$ into the ordered partition of
$\{0,1,2,3\}$, with $\{0,1,2,3\}-A$ as the first block and $A$ as
the second block.

\begin{center}
\begin{picture}(300,330)(0,-15)

\put(0,0){\circle*{2}} \put(300,0){\circle*{2}}
\put(0,300){\circle*{2}} \put(300,300){\circle*{2}}
\put(40,40){\circle*{2}} \put(260,40){\circle*{2}}
\put(40,260){\circle*{2}} \put(260,260){\circle*{2}}

\put(120,60){\circle*{2}} \put(180,60){\circle*{2}}
\put(100,100){\circle*{2}} \put(200,100){\circle*{2}}
\put(60,120){\circle*{2}} \put(240,120){\circle*{2}}
\put(130,130){\circle*{2}} \put(170,130){\circle*{2}}
\put(130,170){\circle*{2}} \put(170,170){\circle*{2}}
\put(60,180){\circle*{2}} \put(240,180){\circle*{2}}
\put(100,200){\circle*{2}} \put(200,200){\circle*{2}}
\put(120,240){\circle*{2}} \put(180,240){\circle*{2}}


\put(150,30){\makebox(0,0){$\{3\}$}}
\put(150,95){\makebox(0,0){$\{1,2,3\}$}}
\put(150,150){\makebox(0,0){$\{1,2\}$}}
\put(150,205){\makebox(0,0){$\{0,1,2\}$}}
\put(150,270){\makebox(0,0){$\{0\}$}}
\put(150,-7){\makebox(0,0){$\{0,3\}$}}
\put(150,307){\makebox(0,0){$\{0,3\}$}}
\put(-13,150){\makebox(0,0){$\{0,3\}$}}
\put(313,150){\makebox(0,0){$\{0,3\}$}}
\put(75,75){\makebox(0,0){$\{2,3\}$}}
\put(225,75){\makebox(0,0){$\{1,3\}$}}
\put(75,225){\makebox(0,0){$\{0,2\}$}}
\put(225,225){\makebox(0,0){$\{0,1\}$}}
\put(30,150){\makebox(0,0){$\{0,2,3\}$}}
\put(95,150){\makebox(0,0){$\{2\}$}}
\put(205,150){\makebox(0,0){$\{1\}$}}
\put(270,150){\makebox(0,0){$\{0,1,3\}$}}

\put(0,-5){\makebox(0,0){\scriptsize $1203$}}
\put(300,-5){\makebox(0,0){\scriptsize $2103$}}
\put(0,305){\makebox(0,0){\scriptsize $1230$}}
\put(300,305){\makebox(0,0){\scriptsize $2130$}}

\put(27,40){\makebox(0,0){\scriptsize $1023$}}
\put(273,40){\makebox(0,0){\scriptsize $2013$}}
\put(120,55){\makebox(0,0){\scriptsize $0123$}}
\put(180,55){\makebox(0,0){\scriptsize $0213$}}
\put(112,100){\makebox(0,0){\scriptsize $0132$}}
\put(188,100){\makebox(0,0){\scriptsize $0231$}}
\put(70,122){\makebox(0,0){\scriptsize $1032$}}
\put(230,122){\makebox(0,0){\scriptsize $2031$}}
\put(120,132){\makebox(0,0){\scriptsize $0312$}}
\put(180,132){\makebox(0,0){\scriptsize $0321$}}
\put(120,168){\makebox(0,0){\scriptsize $3012$}}
\put(180,168){\makebox(0,0){\scriptsize $3021$}}
\put(70,178){\makebox(0,0){\scriptsize $1302$}}
\put(230,178){\makebox(0,0){\scriptsize $2301$}}
\put(112,200){\makebox(0,0){\scriptsize $3102$}}
\put(188,200){\makebox(0,0){\scriptsize $3201$}}
\put(120,245){\makebox(0,0){\scriptsize $3120$}}
\put(180,245){\makebox(0,0){\scriptsize $3210$}}
\put(27,260){\makebox(0,0){\scriptsize $1320$}}
\put(273,260){\makebox(0,0){\scriptsize $2310$}}


\put(0,0){\line(1,0){300}} \put(0,0){\line(0,1){300}}
\put(300,0){\line(0,1){300}} \put(0,300){\line(1,0){300}}
\put(0,0){\line(1,1){40}} \put(300,0){\line(-1,1){40}}
\put(0,300){\line(1,-1){40}} \put(300,300){\line(-1,-1){40}}

\put(40,40){\line(4,1){80}} \put(40,40){\line(1,4){20}}
\put(260,40){\line(-4,1){80}} \put(260,40){\line(-1,4){20}}
\put(40,260){\line(4,-1){80}} \put(40,260){\line(1,-4){20}}
\put(260,260){\line(-4,-1){80}} \put(260,260){\line(-1,-4){20}}

\put(100,100){\line(-2,1){40}} \put(100,100){\line(1,-2){20}}
\put(200,100){\line(2,1){40}} \put(200,100){\line(-1,-2){20}}
\put(100,200){\line(-2,-1){40}} \put(100,200){\line(1,2){20}}
\put(200,200){\line(2,-1){40}} \put(200,200){\line(-1,2){20}}

\put(120,60){\line(1,0){60}} \put(120,240){\line(1,0){60}}
\put(60,120){\line(0,1){60}} \put(240,120){\line(0,1){60}}

\put(100,100){\line(1,1){30}} \put(200,100){\line(-1,1){30}}
\put(100,200){\line(1,-1){30}} \put(200,200){\line(-1,-1){30}}

\put(130,130){\line(1,0){40}} \put(130,130){\line(0,1){40}}
\put(170,170){\line(-1,0){40}} \put(170,170){\line(0,-1){40}}

\put(105,50){\textcolor{red}{\circle*{2}}}
\put(100,60){\textcolor{red}{\circle*{2}}}
\put(145,60){\textcolor{red}{\circle*{2}}}
\put(110,70){\textcolor{red}{\circle*{2}}}
\put(120,75){\textcolor{red}{\circle*{2}}}
\put(100,60){\textcolor{red}{\line(1,-2){5}}}
\put(100,60){\textcolor{red}{\line(1,1){10}}}
\put(110,70){\textcolor{red}{\line(2,1){10}}}
\put(145,60){\textcolor{red}{\line(-5,3){25}}}
\put(145,60){\textcolor{red}{\line(-4,-1){40}}}
\put(100,90){\textcolor{red}{\circle*{2}}}
\put(90,100){\textcolor{red}{\circle*{2}}}
\put(110,95){\textcolor{red}{\circle*{2}}}
\put(110,110){\textcolor{red}{\circle*{2}}}
\put(95,110){\textcolor{red}{\circle*{2}}}
\put(100,90){\textcolor{red}{\line(2,1){10}}}
\put(100,90){\textcolor{red}{\line(-1,1){10}}}
\put(110,110){\textcolor{red}{\line(-1,0){15}}}
\put(110,110){\textcolor{red}{\line(0,-1){15}}}
\put(90,100){\textcolor{red}{\line(1,2){5}}}
\put(23.5,23.5){\textcolor{red}{\circle*{2}}}
\put(65,40){\textcolor{red}{\circle*{2}}}
\put(60,50){\textcolor{red}{\circle*{2}}}
\put(50,60){\textcolor{red}{\circle*{2}}}
\put(40,65){\textcolor{red}{\circle*{2}}}
\put(65,40){\textcolor{red}{\line(-5,-2){41}}}
\put(65,40){\textcolor{red}{\line(-1,2){5}}}
\put(40,65){\textcolor{red}{\line(-2,-5){16.5}}}
\put(40,65){\textcolor{red}{\line(2,-1){10}}}
\put(50,60){\textcolor{red}{\line(1,-1){10}}}
\put(50,105){\textcolor{red}{\circle*{2}}}
\put(60,100){\textcolor{red}{\circle*{2}}}
\put(60,145){\textcolor{red}{\circle*{2}}}
\put(70,110){\textcolor{red}{\circle*{2}}}
\put(75,120){\textcolor{red}{\circle*{2}}}
\put(60,100){\textcolor{red}{\line(-2,1){10}}}
\put(60,100){\textcolor{red}{\line(1,1){10}}}
\put(70,110){\textcolor{red}{\line(1,2){5}}}
\put(60,145){\textcolor{red}{\line(3,-5){15}}}
\put(60,145){\textcolor{red}{\line(-1,-4){10}}}
\put(120,120){\textcolor{red}{\circle*{2}}}
\put(145,125){\textcolor{red}{\circle*{2}}}
\put(125,145){\textcolor{red}{\circle*{2}}}
\put(145,135){\textcolor{red}{\circle*{2}}}
\put(135,145){\textcolor{red}{\circle*{2}}}
\put(120,120){\textcolor{red}{\line(5,1){25}}}
\put(120,120){\textcolor{red}{\line(1,5){5}}}
\put(145,135){\textcolor{red}{\line(0,-1){10}}}
\put(145,135){\textcolor{red}{\line(-1,1){10}}}
\put(125,145){\textcolor{red}{\line(1,0){10}}}
\put(-5,20){\textcolor{red}{\circle*{2}}}
\put(5,20){\textcolor{red}{\circle*{2}}}
\put(20,-5){\textcolor{red}{\circle*{2}}}
\put(20,5){\textcolor{red}{\circle*{2}}}
\put(15,15){\textcolor{red}{\circle*{2}}}
\put(-5,20){\textcolor{red}{\line(1,-1){25}}}
\put(-5,20){\textcolor{red}{\line(1,0){10}}}
\put(15,15){\textcolor{red}{\line(-2,1){10}}}
\put(15,15){\textcolor{red}{\line(1,-2){5}}}
\put(20,-5){\textcolor{red}{\line(0,1){10}}}

\put(65,40){\textcolor{red}{\line(4,1){40}}}
\put(60,50){\textcolor{red}{\line(4,1){40}}}
\put(40,65){\textcolor{red}{\line(1,4){10}}}
\put(50,60){\textcolor{red}{\line(1,4){10}}}
\put(70,110){\textcolor{red}{\line(2,-1){20}}}
\put(75,120){\textcolor{red}{\line(2,-1){20}}}
\put(100,90){\textcolor{red}{\line(1,-2){10}}}
\put(110,95){\textcolor{red}{\line(1,-2){10}}}

\put(180,120){\textcolor{red}{\circle*{2}}}
\put(155,125){\textcolor{red}{\circle*{2}}}
\put(165,145){\textcolor{red}{\circle*{2}}}
\put(155,135){\textcolor{red}{\circle*{2}}}
\put(175,145){\textcolor{red}{\circle*{2}}}
\put(180,120){\textcolor{red}{\line(-5,1){25}}}
\put(180,120){\textcolor{red}{\line(-1,5){5}}}
\put(155,135){\textcolor{red}{\line(0,-1){10}}}
\put(155,135){\textcolor{red}{\line(1,1){10}}}
\put(165,145){\textcolor{red}{\line(1,0){10}}}

\put(180,180){\textcolor{red}{\circle*{2}}}
\put(155,165){\textcolor{red}{\circle*{2}}}
\put(165,155){\textcolor{red}{\circle*{2}}}
\put(155,175){\textcolor{red}{\circle*{2}}}
\put(175,155){\textcolor{red}{\circle*{2}}}
\put(180,180){\textcolor{red}{\line(-5,-1){25}}}
\put(180,180){\textcolor{red}{\line(-1,-5){5}}}
\put(155,175){\textcolor{red}{\line(0,-1){10}}}
\put(155,165){\textcolor{red}{\line(1,-1){10}}}
\put(165,155){\textcolor{red}{\line(1,0){10}}}

\put(120,180){\textcolor{red}{\circle*{2}}}
\put(145,165){\textcolor{red}{\circle*{2}}}
\put(125,155){\textcolor{red}{\circle*{2}}}
\put(145,175){\textcolor{red}{\circle*{2}}}
\put(135,155){\textcolor{red}{\circle*{2}}}
\put(120,180){\textcolor{red}{\line(5,-1){25}}}
\put(120,180){\textcolor{red}{\line(1,-5){5}}}
\put(145,175){\textcolor{red}{\line(0,-1){10}}}
\put(135,155){\textcolor{red}{\line(1,1){10}}}
\put(125,155){\textcolor{red}{\line(1,0){10}}}

\put(100,210){\textcolor{red}{\circle*{2}}}
\put(90,200){\textcolor{red}{\circle*{2}}}
\put(110,205){\textcolor{red}{\circle*{2}}}
\put(110,190){\textcolor{red}{\circle*{2}}}
\put(95,190){\textcolor{red}{\circle*{2}}}
\put(110,205){\textcolor{red}{\line(-2,1){10}}}
\put(90,200){\textcolor{red}{\line(1,1){10}}}
\put(110,190){\textcolor{red}{\line(-1,0){15}}}
\put(110,190){\textcolor{red}{\line(0,1){15}}}
\put(95,190){\textcolor{red}{\line(-1,2){5}}}

\put(200,210){\textcolor{red}{\circle*{2}}}
\put(210,200){\textcolor{red}{\circle*{2}}}
\put(190,205){\textcolor{red}{\circle*{2}}}
\put(190,190){\textcolor{red}{\circle*{2}}}
\put(205,190){\textcolor{red}{\circle*{2}}}
\put(200,210){\textcolor{red}{\line(-2,-1){10}}}
\put(200,210){\textcolor{red}{\line(1,-1){10}}}
\put(190,190){\textcolor{red}{\line(1,0){15}}}
\put(190,190){\textcolor{red}{\line(0,1){15}}}
\put(205,190){\textcolor{red}{\line(1,2){5}}}

\put(210,100){\textcolor{red}{\circle*{2}}}
\put(200,90){\textcolor{red}{\circle*{2}}}
\put(190,95){\textcolor{red}{\circle*{2}}}
\put(190,110){\textcolor{red}{\circle*{2}}}
\put(205,110){\textcolor{red}{\circle*{2}}}
\put(200,90){\textcolor{red}{\line(1,1){10}}}
\put(200,90){\textcolor{red}{\line(-2,1){10}}}
\put(190,110){\textcolor{red}{\line(0,-1){15}}}
\put(190,110){\textcolor{red}{\line(1,0){15}}}
\put(205,110){\textcolor{red}{\line(1,-2){5}}}

\put(155,60){\textcolor{red}{\circle*{2}}}
\put(200,60){\textcolor{red}{\circle*{2}}}
\put(195,50){\textcolor{red}{\circle*{2}}}
\put(180,75){\textcolor{red}{\circle*{2}}}
\put(190,70){\textcolor{red}{\circle*{2}}}
\put(200,60){\textcolor{red}{\line(-1,-2){5}}}
\put(200,60){\textcolor{red}{\line(-1,1){10}}}
\put(190,70){\textcolor{red}{\line(-2,1){10}}}
\put(155,60){\textcolor{red}{\line(5,3){25}}}
\put(155,60){\textcolor{red}{\line(4,-1){40}}}

\put(145,240){\textcolor{red}{\circle*{2}}}
\put(100,240){\textcolor{red}{\circle*{2}}}
\put(105,250){\textcolor{red}{\circle*{2}}}
\put(110,230){\textcolor{red}{\circle*{2}}}
\put(120,225){\textcolor{red}{\circle*{2}}}
\put(100,240){\textcolor{red}{\line(1,2){5}}}
\put(100,240){\textcolor{red}{\line(1,-1){10}}}
\put(110,230){\textcolor{red}{\line(2,-1){10}}}
\put(145,240){\textcolor{red}{\line(-5,-3){25}}}
\put(145,240){\textcolor{red}{\line(-4,1){40}}}

\put(155,240){\textcolor{red}{\circle*{2}}}
\put(200,240){\textcolor{red}{\circle*{2}}}
\put(195,250){\textcolor{red}{\circle*{2}}}
\put(190,230){\textcolor{red}{\circle*{2}}}
\put(180,225){\textcolor{red}{\circle*{2}}}
\put(200,240){\textcolor{red}{\line(-1,2){5}}}
\put(200,240){\textcolor{red}{\line(-1,-1){10}}}
\put(190,230){\textcolor{red}{\line(-2,-1){10}}}
\put(155,240){\textcolor{red}{\line(5,-3){25}}}
\put(155,240){\textcolor{red}{\line(4,1){40}}}

\put(225,120){\textcolor{red}{\circle*{2}}}
\put(250,105){\textcolor{red}{\circle*{2}}}
\put(240,100){\textcolor{red}{\circle*{2}}}
\put(240,145){\textcolor{red}{\circle*{2}}}
\put(230,110){\textcolor{red}{\circle*{2}}}
\put(240,100){\textcolor{red}{\line(2,1){10}}}
\put(240,100){\textcolor{red}{\line(-1,1){10}}}
\put(230,110){\textcolor{red}{\line(-1,2){5}}}
\put(240,145){\textcolor{red}{\line(-3,-5){15}}}
\put(240,145){\textcolor{red}{\line(1,-4){10}}}

\put(50,195){\textcolor{red}{\circle*{2}}}
\put(60,200){\textcolor{red}{\circle*{2}}}
\put(60,155){\textcolor{red}{\circle*{2}}}
\put(70,190){\textcolor{red}{\circle*{2}}}
\put(75,180){\textcolor{red}{\circle*{2}}}
\put(60,200){\textcolor{red}{\line(-2,-1){10}}}
\put(60,200){\textcolor{red}{\line(1,-1){10}}}
\put(70,190){\textcolor{red}{\line(1,-2){5}}}
\put(60,155){\textcolor{red}{\line(3,5){15}}}
\put(60,155){\textcolor{red}{\line(-1,4){10}}}

\put(250,195){\textcolor{red}{\circle*{2}}}
\put(240,200){\textcolor{red}{\circle*{2}}}
\put(240,155){\textcolor{red}{\circle*{2}}}
\put(230,190){\textcolor{red}{\circle*{2}}}
\put(225,180){\textcolor{red}{\circle*{2}}}
\put(240,200){\textcolor{red}{\line(2,-1){10}}}
\put(240,200){\textcolor{red}{\line(-1,-1){10}}}
\put(230,190){\textcolor{red}{\line(-1,-2){5}}}
\put(240,155){\textcolor{red}{\line(-3,5){15}}}
\put(240,155){\textcolor{red}{\line(1,4){10}}}

\put(276.5,23.5){\textcolor{red}{\circle*{2}}}
\put(235,40){\textcolor{red}{\circle*{2}}}
\put(240,50){\textcolor{red}{\circle*{2}}}
\put(250,60){\textcolor{red}{\circle*{2}}}
\put(260,65){\textcolor{red}{\circle*{2}}}
\put(235,40){\textcolor{red}{\line(5,-2){41}}}
\put(235,40){\textcolor{red}{\line(1,2){5}}}
\put(260,65){\textcolor{red}{\line(2,-5){16.5}}}
\put(260,65){\textcolor{red}{\line(-2,-1){10}}}
\put(250,60){\textcolor{red}{\line(-1,-1){10}}}

\put(23.5,276.5){\textcolor{red}{\circle*{2}}}
\put(65,260){\textcolor{red}{\circle*{2}}}
\put(60,250){\textcolor{red}{\circle*{2}}}
\put(50,240){\textcolor{red}{\circle*{2}}}
\put(40,235){\textcolor{red}{\circle*{2}}}
\put(65,260){\textcolor{red}{\line(-5,2){41}}}
\put(65,260){\textcolor{red}{\line(-1,-2){5}}}
\put(40,235){\textcolor{red}{\line(-2,5){16.5}}}
\put(40,235){\textcolor{red}{\line(2,1){10}}}
\put(50,240){\textcolor{red}{\line(1,1){10}}}

\put(276.5,276.5){\textcolor{red}{\circle*{2}}}
\put(235,260){\textcolor{red}{\circle*{2}}}
\put(240,250){\textcolor{red}{\circle*{2}}}
\put(250,240){\textcolor{red}{\circle*{2}}}
\put(260,235){\textcolor{red}{\circle*{2}}}
\put(235,260){\textcolor{red}{\line(5,2){41}}}
\put(235,260){\textcolor{red}{\line(1,-2){5}}}
\put(260,235){\textcolor{red}{\line(2,5){16.5}}}
\put(260,235){\textcolor{red}{\line(-2,1){10}}}
\put(250,240){\textcolor{red}{\line(-1,1){10}}}

\put(305,20){\textcolor{red}{\circle*{2}}}
\put(295,20){\textcolor{red}{\circle*{2}}}
\put(280,-5){\textcolor{red}{\circle*{2}}}
\put(280,5){\textcolor{red}{\circle*{2}}}
\put(285,15){\textcolor{red}{\circle*{2}}}
\put(305,20){\textcolor{red}{\line(-1,-1){25}}}
\put(305,20){\textcolor{red}{\line(-1,0){10}}}
\put(285,15){\textcolor{red}{\line(2,1){10}}}
\put(285,15){\textcolor{red}{\line(-1,-2){5}}}
\put(280,-5){\textcolor{red}{\line(0,1){10}}}

\put(305,280){\textcolor{red}{\circle*{2}}}
\put(295,280){\textcolor{red}{\circle*{2}}}
\put(280,305){\textcolor{red}{\circle*{2}}}
\put(280,295){\textcolor{red}{\circle*{2}}}
\put(285,285){\textcolor{red}{\circle*{2}}}
\put(305,280){\textcolor{red}{\line(-1,1){25}}}
\put(305,280){\textcolor{red}{\line(-1,0){10}}}
\put(285,285){\textcolor{red}{\line(2,-1){10}}}
\put(285,285){\textcolor{red}{\line(-1,2){5}}}
\put(280,305){\textcolor{red}{\line(0,-1){10}}}

\put(-5,280){\textcolor{red}{\circle*{2}}}
\put(5,280){\textcolor{red}{\circle*{2}}}
\put(20,305){\textcolor{red}{\circle*{2}}}
\put(20,295){\textcolor{red}{\circle*{2}}}
\put(15,285){\textcolor{red}{\circle*{2}}}
\put(-5,280){\textcolor{red}{\line(1,1){25}}}
\put(-5,280){\textcolor{red}{\line(1,0){10}}}
\put(15,285){\textcolor{red}{\line(-2,-1){10}}}
\put(15,285){\textcolor{red}{\line(1,2){5}}}
\put(20,305){\textcolor{red}{\line(0,-1){10}}}

\put(260,65){\textcolor{red}{\line(-1,4){10}}}
\put(250,60){\textcolor{red}{\line(-1,4){10}}}
\put(260,235){\textcolor{red}{\line(-1,-4){10}}}
\put(250,240){\textcolor{red}{\line(-1,-4){10}}}
\put(40,235){\textcolor{red}{\line(1,-4){10}}}
\put(50,240){\textcolor{red}{\line(1,-4){10}}}

\put(235,40){\textcolor{red}{\line(-4,1){40}}}
\put(240,50){\textcolor{red}{\line(-4,1){40}}}
\put(235,260){\textcolor{red}{\line(-4,-1){40}}}
\put(240,250){\textcolor{red}{\line(-4,-1){40}}}
\put(65,260){\textcolor{red}{\line(4,-1){40}}}
\put(60,250){\textcolor{red}{\line(4,-1){40}}}

\put(230,110){\textcolor{red}{\line(-2,-1){20}}}
\put(225,120){\textcolor{red}{\line(-2,-1){20}}}
\put(230,190){\textcolor{red}{\line(-2,1){20}}}
\put(225,180){\textcolor{red}{\line(-2,1){20}}}
\put(70,190){\textcolor{red}{\line(2,1){20}}}
\put(75,180){\textcolor{red}{\line(2,1){20}}}

\put(200,90){\textcolor{red}{\line(-1,-2){10}}}
\put(190,95){\textcolor{red}{\line(-1,-2){10}}}
\put(200,210){\textcolor{red}{\line(-1,2){10}}}
\put(190,205){\textcolor{red}{\line(-1,2){10}}}
\put(100,210){\textcolor{red}{\line(1,2){10}}}
\put(110,205){\textcolor{red}{\line(1,2){10}}}

\put(125,155){\textcolor{red}{\line(0,-1){10}}}
\put(135,155){\textcolor{red}{\line(0,-1){10}}}
\put(145,165){\textcolor{red}{\line(1,0){10}}}
\put(145,175){\textcolor{red}{\line(1,0){10}}}
\put(165,155){\textcolor{red}{\line(0,-1){10}}}
\put(175,155){\textcolor{red}{\line(0,-1){10}}}
\put(145,125){\textcolor{red}{\line(1,0){10}}}
\put(145,135){\textcolor{red}{\line(1,0){10}}}

\put(-5,20){\textcolor{red}{\line(0,1){260}}}
\put(5,20){\textcolor{red}{\line(0,1){260}}}
\put(295,20){\textcolor{red}{\line(0,1){260}}}
\put(305,20){\textcolor{red}{\line(0,1){260}}}
\put(20,-5){\textcolor{red}{\line(1,0){260}}}
\put(20,5){\textcolor{red}{\line(1,0){260}}}
\put(20,295){\textcolor{red}{\line(1,0){260}}}
\put(20,305){\textcolor{red}{\line(1,0){260}}}

\end{picture}
\end{center}

A vertex of the permutohedron is labelled by the permutation
corresponding to the maximal 0-nested set, i.e.\ to the set of all
the facets incident to this vertex. The facets of the
permutoassociahedron correspond to the elements of
$\mathcal{B}_1$. For example, the pentagon that covers the vertex
labelled by 2301 corresponds to the set
$\{\{0,1,3\},\{0,1\},\{1\}\}$.

\begin{center}\textmd{\textbf{Acknowledgements} }
\end{center}
\medskip
We are grateful to Pierre-Louis Curien for a suggestion how to
improve the notation in our paper. We are grateful to G\" unter
M.\ Ziegler for some useful comments, including
Remark~\ref{ziegler}, and for some corrections concerning a
previous version of this paper. The first author is grateful to
the Vietnam Institute for Advanced Studies in Mathematics in Hanoi
for the hospitality, excellent working conditions and generous
financial support while finishing this research. The first, the
second and the third author were partially supported by the Grants
$174020$, III$44006$ and $174026$ of the Ministry for Education
and Science of the Republic of Serbia, respectively.

\end{document}